\documentclass[a4paper, 11pt]{amsart}

\usepackage{amsmath, amsfonts, amssymb}
\usepackage{amsthm}
\usepackage{esint}
\usepackage{tikz} 
\usetikzlibrary{matrix}
\usepackage[colorlinks,bookmarks=false]{hyperref}
\usepackage{cite}
\usepackage{array}
\usepackage{breqn}
\usepackage{mathrsfs}
\usepackage{amsaddr}
\usepackage{enumerate}
\usepackage{pgfplots}
\usepackage [english]{babel}
\usepackage [autostyle, english = american]{csquotes}
\MakeOuterQuote{"}

\numberwithin{equation}{section}

\title{Sharp subelliptic estimates in the $\dbar$-Neumann problem via an uncertainty principle}

\author{Gian Maria Dall'Ara}
\address{Istituto Nazionale di Alta Matematica ``F. Severi"\\ Research Unit Scuola Normale Superiore\\
Piazza dei Cavalieri, 7, 56126, Pisa (Italy)}
\email{dallara@altamatematica.it}

\author{Samuele Mongodi}
\address{Università degli Studi di Milano-Bicocca\\
	Dipartimento di Matematica e Applicazioni\\
	via Roberto Cozzi, 55, 20125 Milano (Italy)}
\email{samuele.mongodi@unimib.it}


\date{\today}

\newcommand{\C}{\mathbb{C}}
\newcommand{\R}{\mathbb{R}}

\newcommand{\N}{\mathbb{N}}

\newcommand{\dbar}{\overline{\partial}}
\newcommand{\E}{\mathbf{E}}
\newcommand{\Dil}{\mathrm{Dil}}
\newcommand{\re}{\mathrm{Re}}
\newcommand{\im}{\mathrm{Im}}
\newcommand{\sfera}{S}

\newtheorem{thm}{Theorem}[section]
\newtheorem{prp}[thm]{Proposition}
\newtheorem{lem}[thm]{Lemma}
\newtheorem{cor}[thm]{Corollary}

\newtheorem{dfn}[thm]{Definition}

\newtheorem{rmk}[thm]{Remark}

\begin{document}

\maketitle

\begin{abstract}
	The problem of giving a (CR-)geometric description of the best possible order of a subelliptic estimate at a boundary point in the $\dbar$-Neumann problem is largely open. In this paper, we introduce a novel technique based on a "$\dbar$-uncertainty principle" and, as an application, we determine the sharp order of subellipticity at the origin for a large class of Kohn's special domains in ambient dimension $\leq 5$. 
	\end{abstract}

\tableofcontents

\section{Introduction}

\subsection{Subellipticity in the $\bar\partial$-Neumann problem}

The $\dbar$-Neumann problem, formulated by Spencer in the 1950s, is the staple of the partial differential equations approach to complex analysis in several variables. We briefly recall its formulation mostly in order to establish notation, referring to \cite{folland_kohn} or \cite{straube_book} for details. 

Given a domain $\Omega\subseteq \C^n$ ($n\geq 2$) and an integer $q\geq 0$, denote by $L^2_{(0,q)}(\Omega)$ the Hilbert space of $(0,q)$-forms, with scalar product defined with respect to the ambient Euclidean metric. The natural differential operator $\dbar$ has a maximal extension $\dbar:L^2_{(0,q)}(\Omega)\rightarrow L^2_{(0,q+1)}(\Omega)$ which, being closed and densely defined, admits a Hilbert space adjoint $\dbar^*:L^2_{(0,q+1)}(\Omega)\rightarrow L^2_{(0,q)}(\Omega)$. We follow the custom of denoting by the same symbol $\dbar$ (resp.~$\dbar^*$) operators acting on forms of varying degrees (equivalently, we think of $\dbar$ and $\dbar^*$ as operators acting on $\oplus_q L^2_{(0,q)}(\Omega)$). 

The Hodge Laplacian \[
\Box:=\dbar^*\dbar+\dbar\dbar^*
\]
is a nonnegative self-adjoint operator preserving form degrees, where the natural domain of definition of $\Box$ is \[\mathrm{dom}(\Box)=\{u\in \mathrm{dom}(\dbar)\cap\mathrm{dom}(\dbar^*)\colon\, \dbar u\in \mathrm{dom}(\dbar^*),\ \dbar^* u\in \mathrm{dom}(\dbar)\}.\] The $\dbar$-Neumann problem (at the level of $(0,q)$-forms) is the noncoercive boundary value problem\begin{equation}\label{eq:dbar_Neumann}
\begin{cases}
	\Box u = v\\
	u\in \mathrm{dom}(\Box)
\end{cases}
\end{equation}
where $v\in L^2_{(0,q)}(\Omega)$ is the datum, and the noncoercive boundary conditions are hidden in the requirement that $u$ lie in the domain of $\Box$, specifically in the conditions $u\in \mathrm{dom}(\dbar^*)$ and $\dbar u\in \mathrm{dom}(\dbar^*)$. 

If $\Omega\subset\C^n$ is \emph{bounded and pseudoconvex} (see \cite[Theorem 2.9]{straube_book} or \cite[Theorem 4.4.1]{chen_shaw}), then $\Box$ has closed range and trivial null space at the level of $(0,q)$-forms for every $q\geq 1$. Hence, in positive degrees the $\dbar$-Neumann problem has a unique solution $u\in L^2_{(0,q)}(\Omega)$ for every datum $v\in L^2_{(0,q)}(\Omega)$. By Hodge theory, this solvability amounts to the vanishing of the $L^2$-Dolbeault cohomology groups in positive degrees. 

Existence of weak solutions of \eqref{eq:dbar_Neumann} is thus elegantly settled, and the next natural question is that of regularity. Since interior regularity follows from classical elliptic theory, in broad terms the problem is to determine \emph{under which assumptions on the boundary $b \Omega$ boundary regularity in the $\dbar$-Neumann problem holds}. This fundamental question gave rise to a variety of deep contributions through the last six decades. We refer the reader to the various existing surveys and accounts, e.g., \cite{boas_straube, chen_shaw, fu_straube, straube_icm, straube_book, catlin_dangelo}. \newline 

\par In this paper, we study the \emph{local regularity problem}, where regularity is measured in the \emph{$L^2$-Sobolev scale}, on pseudoconvex domains $\Omega$ with smooth boundary. From now on, we \emph{restrict our considerations to the level of $(0,1)$-forms} (our methods should be extendable to forms of higher degree, but we wish to keep at a minimum the complexity of the setting). 

A crucial notion in the local regularity theory of the $\dbar$-Neumann problem is that of a \emph{subelliptic estimate}, which we proceed to recall. Let $p\in b\Omega$ be a boundary point and $s>0$. A \emph{subelliptic estimate of order $s$} is said to hold at $p$ if there exist a constant $C>0$ and a neighborhood $V$ of $p$ in $\C^n$ such that the inequality 
\begin{equation}\label{eq:subelliptic}
	\lVert u\rVert_s^2\leq C(	\lVert\dbar u\rVert^2 + \lVert\dbar^*u\rVert^2+\lVert u\rVert^2)
\end{equation}
holds for all $(0,1)$-forms $u$ in the domain of $\dbar^*$ having coefficients in $C^\infty_c(\overline{\Omega})$ and supported in $\overline\Omega\cap V$. Here $\lVert\cdot\rVert_s$ is the $L^2$-Sobolev norm of order $s$, and $\lVert\cdot\rVert=\lVert\cdot\rVert_0$ is the $L^2$ norm. 

The interest in \eqref{eq:subelliptic} stems from the fact that it implies a local gain of boundary regularity in the $L^2$-Sobolev scale for the $\dbar$-Neumann problem. More precisely, assume that $\Omega$ is bounded, pseudoconvex and has a smooth boundary, and denote by $W^t_{(0,1)}(U)\subseteq L^2_{(0,1)}(U)$ the $L^2$-Sobolev space of order $t\geq 0$ on the open set $U$. If a subelliptic estimate of order $s$ holds at $p$, then we have the following: for every datum $v\in L^2_{(0,1)}(\Omega)$ such that $v_{|\Omega\cap U} \in W^t_{(0,1)}(\Omega\cap U)$, where $U$ is a neighborhood of $p$ in $\C^n$, we have $u_{|\Omega\cap U'}\in W^{t+2s}_{(0,1)}(\Omega\cap U')$, where $u$ is the solution of \eqref{eq:dbar_Neumann} and $U'$ is a possibly smaller neighborhood of $p$. We refer the reader to Theorem 1.13 of \cite{kohn_1979} for a more detailed statement, including implications for the canonical solution of the $\dbar$-problem and for the Bergman projection. 

A few basic and well-known remarks about subelliptic estimates are in order: \begin{enumerate}
	\item The validity of a subelliptic estimate is a local property of the boundary $b \Omega$. This is trivial, because the condition $u\in \mathrm{dom}(\dbar^*)$ is local (as recalled at the beginning of Section \ref{sec:reduction_boundary}). Thus, whether a subelliptic estimate of a given order $s$ is valid at $p$ only depends on the germ of $b \Omega$ at $p$. 
	\item The validity of \eqref{eq:subelliptic} is independent of the Hermitian metric chosen to define the $L^2$ spaces of forms and the adjoint $\dbar^*$ (see \cite[Theorem 1]{sweeney} and the remarks following its statement, or \cite[Theorem 2]{celik_straube}). In particular, any (locally defined) Hermitian metric may be used in place of the standard Euclidean metric of $\C^n$. We will take advantage of this freedom in Section \ref{sec:pseudoconvex_rigid_domains} below. 
	\item As a consequence of remarks (1) and (2), the validity of a subelliptic estimate of order $s$ at $p$ is invariant under local biholomorphisms. In slightly imprecise terms, it solely depends on the CR geometry of the germ $(b \Omega, p)$. 
	\end{enumerate}

\subsection{The sharp order of subellipticity. Previous work}

Our focus in the present paper is on the \emph{best possible order of a subelliptic estimate}. It is convenient to give the following definition. 

\begin{dfn}[Sharp order of subellipticity]
	Given a germ of pseudoconvex domain $(\Omega,p)$ with smooth boundary (where $p\in b\Omega$), we define the sharp order of subellipticity as \[
	s(\Omega, p):=\sup\{s\geq0\colon\ \text{\eqref{eq:subelliptic} holds for some choice of $C$ and $V$}\}.
	\]
\end{dfn}

Notice that the set of exponents $s\geq 0$ for which a subelliptic estimate of order $s$ holds at $p$ is either $[0,s(\Omega, p)]$ or $[0,s(\Omega, p))$, and there are examples of the latter possibility in dimension three or higher (see \cite[Theorem 7.2]{catlin_dangelo}). \newline

\par
If $\Omega$ is strongly pseudoconvex at $p$, then $s(\Omega, p)= \frac{1}{2}$ (and the supremum is achieved) \cite{kohn_harmonic_I, kohn_harmonic_II}. \newline The invariant $s(\Omega, p)$ is completely understood in dimension $2$. If $\Omega\subseteq \C^2$, then \[
	s(\Omega, p)=\frac{1}{2m}
	\]
(and the supremum is achieved), where the quantity $2m$ is an even integer, or $+\infty$ (in which case no subelliptic estimate holds), admitting various equivalent descriptions. E.g., it equals the maximum order of contact of smooth complex analytic curves with $b \Omega$ at $p$, and it also equals the minimum number of iterated commutators of horizontal vector fields required to span the tangent space of $b \Omega$ at $p$, where horizontal refers to the natural CR structure on the boundary. This result is due to Kohn \cite{kohn_dim2} and Greiner \cite{greiner}. See also \cite[Theorem 3.1]{catlin_dangelo}. 

In dimension three or higher, the present understanding of the invariant $s(\Omega, p)$ is much more limited. The basic qualitative problem of deciding whether $s(\Omega,p)$ is positive, that is, whether a subelliptic estimate holds at $p$, has been settled by Catlin \cite{catlin_necessary, catlin_boundary, catlin_subelliptic}: we have \begin{equation}\label{eq:catlin_characterization} 
	s(\Omega, p)>0 \quad \text{if and only if} \quad \Delta^1(b \Omega, p)<+\infty, \end{equation}
where $\Delta^1(b \Omega, p)$ is the D'Angelo $1$-type of $(b\Omega, p)$, defined as the maximum order of contact of one-dimensional, \emph{possibly singular}, complex analytic varieties with $b\Omega$ at $p$ (see \cite{dangelo_type, dangelo_book} for this notion of type). More precisely, Catlin states in \cite[p. 133]{catlin_subelliptic} that $
s(\Omega, p)\geq \frac{1}{\Delta^1(b\Omega, p)^{n^2\Delta^1(b\Omega, p)^{n^2}}}$, 
while \cite[Theorem 3]{catlin_necessary} contains the neater upper bound $s(\Omega, p)\leq \frac{1}{\Delta^1(b\Omega, p)}$ (actually, a stronger bound is proven, see Theorem \ref{thm:catlin_necessity} below).

A few years before Catlin's work, Kohn developed a very different approach to the subellipticity problem. In  \cite{kohn_1979}, he introduced subelliptic multiplier ideals, namely ideals $I_s(p)$ consisting of germs $f$ of smooth functions at $p$ having the property that $	\lVert f u\rVert_s^2\leq C(	\lVert\dbar u\rVert^2 + \lVert\dbar^*u\rVert^2+\lVert u\rVert^2)$ for the same space of forms $u$ as in \eqref{eq:subelliptic}. He devised a number of procedures of an algebraic nature allowing to pass from an element $f\in I_s(p)$ to a related element $g\in I_{s'}(p)$, where $s'$ is smaller (in a controlled way) than $s$. If by a suitable sequence of such procedures one is able to obtain that $1\in I_s(p)$, then the inequality $s(\Omega,p)\geq s$ is established. This approach to the subellipticity problem turns out to be successful when the boundary of $\Omega$ is real-analytic near $p$ (thanks to a result of Diederich--Fornæss \cite{diederich_fornaess}): in this case, $s(\Omega, p)>0$ if and only if the boundary does not contain any germ of positive dimensional complex analytic variety near $p$, a condition that, in the real-analytic case, is equivalent to the finiteness of $\Delta^1(b\Omega, p)$. Unfortunately, the problem of generalizing Kohn's multiplier ideal approach beyond the real-analytic category appears to be very difficult, and it remains an open question whether a suitable variant of Kohn's method may be used to prove \eqref{eq:catlin_characterization}. \newline

\par The quantitative lower bounds on $s(\Omega, p)$ (in terms of geometric invariants of the germ of $b\Omega$ at $p$) that can be extracted from the results of Kohn and Catlin are quite poor. Indeed, there is a substantial gap between Catlin's upper and lower bounds, while Kohn's original approach fails completely to produce any lower bound on $s(\Omega, p)$ depending only on the dimension and the D'Angelo $1$-type (see, e.g., \cite[Proposition 4.4]{catlin_dangelo}). The problem of obtaining such estimates via a refinement of Kohn's multiplier ideal method has been revived recently by works of Siu \cite{siu_effective, siu_new} and Kim--Zaitsev \cite{kim_zaitsev_Q, kim_zaitsev_jet, kim_zaitsev_triangular}, who in particular solved it for the important class of so-called \emph{special domains} \begin{equation}\label{eq:special_domain}
\Omega=\{(z,z_{n+1})\in \C^{n+1}\colon\ \im(z_{n+1})>|F(z)|^2\}\subset \C^{n+1}, 
\end{equation}
where $F:\C^n\rightarrow \C^m$ is holomorphic. The lower bound obtained by Kim--Zaitsev is of the form $s(\Omega, 0)\geq e^{-C_n \Delta^1(b\Omega, 0)^{C_n}}$ (see \cite[Theorem 1.6]{kim_zaitsev_triangular} and the remark that follows).

The exponential lower bounds on the sharp order of subellipticity in terms of the D'Angelo $1$-type obtained by Catlin (and by Kim--Zaitsev in the case of special domains) are of course very far from being sharp in dimension $2$ (cf.~the results by Kohn and Greiner cited above), and this is expected to be the case in higher dimensions too. In fact, D'Angelo conjectured \cite[p.205]{dangelo_book} that $s(\Omega, p)\geq \frac{1}{B^1(b\Omega, p)}$, where $B^1(b\Omega, p)$ is a further notion of $1$-type taking integer values in the interval $\left[\Delta^1(b\Omega, p), 2^{2-n}\Delta^1(b\Omega, p)^{n-1}\right]$. 

Finally, let us point out the following recent papers on subellipticity in the $\dbar$-Neumann problem: \cite{zimmer}, where metric geometry methods are combined with Catlin's potential theory on convex domains; a recent preprint by Siu \cite{siu_allvariables}, where he treats a more general class of domains of the form \eqref{eq:special_domain} where $F$ is allowed to depend on the last variable $z_{n+1}$ too; the paper \cite{kohn_mexico}, where "modified" Kohn's algorithms are discussed.

\subsection{Content of the paper}

In view of the above discussion, it would be desirable to have a technique to establish subelliptic estimates, alternative to both Catlin's potential theoretic approach and Kohn's multiplier ideal machinery, capable of capturing more precise information about the invariant $s(\Omega, p)$. 

Motivated by these considerations, we introduce below a new method, based on a \emph{$\dbar$-uncertainty principle}, in the study of subellipticity of the $\dbar$-Neumann problem. As an application of our method, Theorem \ref{thm:main} below \emph{determines the sharp order of subellipticity at the origin for a large class of "dilation-invariant special domains" of dimension $\leq 5$}, defined in Section \ref{sec:ker_one_dim}. We believe that this result is satisfactory in two particular respects: \begin{enumerate}
\item The domains considered exhibit a rich geometry, e.g., the D'Angelo type may fail to be upper semicontinuous (see Section \ref{sec:examples}), which is a well-known difficulty appearing in dimension three and higher.  
\item The sharp order of subellipticity has a very neat geometric description and it turns out to be equal to the reciprocal of the upper semicontinuous envelop of a "$1$-type" implicitly appearing in the paper \cite{catlin_necessary} cited above. 
	\end{enumerate}

Let us remark that we believe that our method allows to establish quite precise bounds for the sharp order of subellipticity of more general domains than those covered by Theorem \ref{thm:main}, but this we have to leave to future investigations. \newline 

\par We now discuss a bit more precisely the content of the paper. Given a \emph{rigid domain}  \[
\Omega=\{(z,z_{n+1})\in \C^{n+1}\colon\ \im(z_{n+1})>\varphi(z)\}\subset \C^{n+1}, 
\]
where $\varphi:\C^n\rightarrow \R$ is a (germ of) smooth plurisubharmonic function vanishing at zero, in Section \ref{sec:rigid} we reduce the problem of bounding from below $s(\Omega, 0)$ to that of proving certain \emph{spectral gap estimates} for a family of quadratic forms $\mathbf{E}^{\xi\varphi}$ defined on test functions on $\C^n$ and depending on $\varphi$ and a "semiclassical" parameter $\xi\rightarrow+\infty$ (Proposition \ref{prp:spectral_gap}). 

In Section \ref{sec:dbar_up}, we elaborate on a $\bar\partial$-uncertainty principle (Lemma \ref{lem:dbar_unc}) originally introduced in \cite{dallara}. Combining it with a sublevel set estimate for one-variable complex polynomials (Lemma \ref{lem:sublevel}), we deduce a \emph{local one-dimensional estimate} designed to analyze the quadratic forms $\mathbf{E}^{\xi\varphi}$ in the case where $\varphi=|F|^2$ with $F:\C^n\rightarrow \C^m$ a polynomial mapping (Proposition \ref{lem:main_1d_estimate}). 

In order to pass from the local one-dimensional estimates of Section \ref{sec:dbar_up} to the $n$-dimensional spectral gap estimates needed for subellipticity, some geometric information about the polynomial mapping $F$ has to be used. A notion that appears to be relevant here is that of an \emph{approximate minimal eigenvector field}, which we discuss in Section \ref{sec:AME}. 

In Section \ref{sec:ker_one_dim} we apply the theory developed in the preceding sections to the proof of the already cited Theorem \ref{thm:main}. The reader may jump right away to Section \ref{sec:ker_one_dim} for the definition of the relevant class of domains and a statement of the theorem, if he so wishes. In Section \ref{sec:examples} we discuss in some detail the nature of the hypotheses of Theorem \ref{thm:main}, and Section \ref{sec:closing} contains some comments on what is left to be understood in the $\dbar$-uncertainty principle approach to subellipticity. Finally, an Appendix contains two elementary estimates on harmonic functions that are used in the proofs of Lemma \ref{lem:dbar_unc_metric} and Proposition \ref{prp:FvsCatlin}. 

The next section contains notation and terminology used systematically in the rest of the paper. 

\subsection{Notation and terminology} We denote by $\mathbb{D}$ the open unit disc in the complex plane and by $D(z,r)$ the disc of center $z\in \C$ and radius $r>0$. We use $\R^+$ for the group of positive real numbers and $\C^\times$ for the group of nonzero complex numbers. The $n$-dimensional complex projective space is denoted by $\mathbb{P}^n$ and the unit sphere in $\C^n$ is denoted by $\sfera^{2n-1}$. 
  
  Integrals with no explicit indication of the measure, e.g. $\int_D f$, are always w.r.t.~Lebesgue measure on the domain of integration $D$. 
  
  The approximate inequalities $A\lesssim_\lambda B$ and $B\gtrsim_\lambda A$ mean that $A\leq KB$ for a constant $K$ that is allowed to depend only on the parameter $\lambda$ and possibly the ambient dimension $n$. We will also occasionally use the big O notation.
  
  The standard Hermitian product and norm in $\C^n$ are denoted $(\cdot, \cdot)$ and $|\cdot|$ respectively. As usual, $\frac{\partial}{\partial z_j}=\frac{1}{2}\left(\frac{\partial}{\partial x_j}-i\frac{\partial}{\partial y_j}\right)$ and $\frac{\partial}{\partial \overline z_j}=\frac{1}{2}\left(\frac{\partial}{\partial x_j}+i\frac{\partial}{\partial y_j}\right)$, where $z_j=x_j+iy_j$ are the standard coordinates on $\C^n$. \newline
  
  \par     Let $U\subseteq \C^n$ be open and let $F_1,\ldots, F_m:U\rightarrow \C$ be holomorphic. Then the real-analytic function \begin{equation}\label{eq:HSOS}
  	\varphi(z)=\sum_{\ell=1}^m|F_\ell(z)|^2
  \end{equation}
  is said to be a \emph{Hermitian Sum of Squares} (in short, HSOS). Any HSOS is plurisubharmonic (see \eqref{eq:J^*J} below). If $m=n$, we say that $\varphi$ is a \emph{Hermitian Sum of $n$ Squares (HSOnS)}.  
  \par Let $U\subseteq\C^n$ be an open neighborhood of $0$ and let $\varphi:U\rightarrow \R$ be smooth and plurisubharmonic (in short, plush). We associate to $\varphi$ the \emph{rigid domain} \[
  \Omega= \{(z,z_{n+1})\in U\times \C\colon\ \im(z_{n+1})>\varphi(z)\}\subset \C^{n+1}. 
  \]   Plurisubharmonicity of $\varphi$ is equivalent to pseudoconvexity of $\Omega$ (at points of the boundary where $z\in U$). We always assume, without loss of generality, that  \begin{equation}\label{eq:gradient=0}
  	\varphi(0)=\frac{\partial\varphi}{\partial z_1}(0)=\ldots=\frac{\partial\varphi}{\partial z_n}(0)=0. 
  \end{equation}  
  This assumption of course means that the origin of $\C^{n+1}$ is on the boundary $b\Omega$ of the domain and the tangent space to $b\Omega$ at the origin is the real hyperplane of equation $\im(z_{n+1})=0$.\newline  
  Finally, we use the symbol $W^s$ for $L^2$-based Sobolev spaces of order $s\in \R$.

   \subsection{Acknowledgment} We would like to thank J.~P.~D'Angelo for his encouragement in pursuing this project, and A.~Vistoli and D.~Zaitsev for helpful comments on specific aspects of this work.

\section{Subellipticity via spectral gap estimates}\label{sec:rigid}

Let $U\subseteq \C^n$ be an open set and let $\varphi:U\rightarrow \R$ be smooth and plush. Denote by \begin{equation}\label{eq:H}
H = \left(\frac{\partial^2\varphi}{\partial \overline{z}_j\partial z_k}\right)_{j,k=1}^n\end{equation}
the Levi form (a.k.a., the complex Hessian) of $\varphi$. Plurisubharmonicity of $\varphi$ amounts to pointwise nonnegativity of the Hermitian matrix-valued function $H$, that is,  \[
(H(z)v,v)=\sum_{j,k=1}^n\frac{\partial^2\varphi}{\partial\overline{z}_j\partial z_k}(z)v_k\overline{v}_j\geq 0\qquad \forall v\in \C^n,\, z\in U.\] 

By the spectral theorem, $H(z)$ is unitarily diagonalizable with nonnegative eigenvalues. We denote by $\lambda_j(z)$ the $j$-th eigenvalue of $H(z)$ w.r.t.~the natural increasing ordering, that is, $\lambda_1(z)\leq \ldots\leq \lambda_n(z)$. Being the zeros of the characteristic polynomial of $H(z)$, which depends continuously on $z$, the eigenvalues $\lambda_j(z)$ are continuous in $z$ too. 

\begin{dfn}\label{dfn:energy_form}
	Let $\varphi:\C^n\rightarrow \R$ be smooth and plush. We define $\E^\varphi$ as the quadratic form defined on functions $w:\C^n\rightarrow \C$ by the formula \begin{equation}\label{energy_form}
		\E^\varphi(w) = \int_{\C^n}|\nabla^{0,1}w|^2e^{-2\varphi} + 2\int_{\C^n}\lambda_1|w|^2e^{-2\varphi}.
	\end{equation}
\end{dfn}
Here $\nabla^{0,1}w =\left(\frac{\partial w}{\partial\overline{z}_1}, \ldots, \frac{\partial w}{\partial\overline{z}_n}\right)$ is the $(0,1)$-part of the gradient of $w$ and $|\nabla^{0,1}w|=\sqrt{\sum_{j=1}^n\left|\frac{\partial w}{\partial \overline{z}_j}\right|^2}$ is its Euclidean norm. The precise domain of definition of the quadratic form $\E^\varphi$ will not be an issue, as we will only need to consider $\E^\varphi(w)$ when $w$ is a test function. The definition is extended to vector-valued functions $w=(w_1,\ldots, w_n)\in C^\infty_c(\C^n; \C^n)$ as follows: \begin{equation}\label{eq:energy_form_vector}
\E^\varphi(w) = \sum_{j=1}^n\E^\varphi(w_j).  
\end{equation}

\begin{rmk}
A close relative of the quadratic form \eqref{energy_form}, namely the energy form of a weighted $\dbar$-Laplacian on $(0,1)$-forms, plays a prominent role in the following papers (among others): \cite{christ, delin, raich, haslinger_helffer, dallara_pointwise, berger_dallara_son}.
	\end{rmk}

The main result of this section is the following. 

\begin{prp}[Spectral gap estimates imply subellipticity]\label{prp:spectral_gap}
	Let $s\leq 1$. Suppose that there exist a positive constant $C\geq 1$ such that the estimate \begin{equation}\label{spectral_gap}
		\E^{\xi\varphi}(u)\geq C^{-1} \xi^{2s}\int_{\C^n}|u|^2e^{-2\xi\varphi}
	\end{equation} holds for every $\xi\geq C$ and $u\in C_c^\infty(\{|z|<C^{-1}\})$. Then the $\dbar$-Neumann problem on the rigid domain \[
\Omega= \{(z,z_{n+1})\in \C^{n+1}\colon\ \im(z_{n+1})>\varphi(z)\}
\] satisfies a subelliptic estimate of order $s$ at the origin. 
\end{prp}

Before turning to the proof of this proposition, let us discuss the behavior of estimate \eqref{spectral_gap} under scalings. Define, for $R>0$,
\[
\Dil_Rw(z) = R^{-n}w(R^{-1}z).
\]
Notice that the real Jacobian determinant of $z\mapsto Rz$ equals $R^{2n}$, and hence $\lVert\Dil_Rw\rVert_{L^2} = \lVert w\lVert_{L^2}$ for every $R>0$.  

\begin{prp}[$\E^{\varphi}$ behaves well under scalings]\label{prp:scaling} Define, for $R>0$, \[
	\varphi_R(z)=\varphi(R^{-1}z).
	\] Then \begin{equation}\label{eq:energy_scaling}
		\E^{\varphi}(w) = R^2 \E^{\varphi_R}(\mathrm{Dil}_Rw).
	\end{equation} As a consequence, if $d>0$ and
	\[s=\frac{1-\delta}{d},\]
	where $\delta\in [0,1]$, then the spectral gap estimate \eqref{spectral_gap} of Proposition \ref{prp:spectral_gap} holds if and only if
	\begin{equation}\label{eq:spectral_gap_rescaled}
		\E^{R^d\varphi_R}(w)\geq C^{-1}R^{-2\delta}\int_{\C^n}|w|^2e^{-R^d\varphi_R}
	\end{equation} holds for every $R\geq C^{\frac{1}{d}}$ and $w\in C^\infty_c(\{|z|\leq C^{-1}R \})$.
\end{prp}

\begin{proof}
	First of all, notice that 
	\begin{eqnarray*}
		|\nabla^{0,1}\left(\Dil_Ru\right)|^2 = R^{-2}|\Dil_R(\nabla^{0,1}u)|^2. 
	\end{eqnarray*}	
	
	Denoting by $H_R$	the Levi form of $\varphi_R$, we have $H_R(z)=R^{-2}H(R^{-1}z)$. If $\lambda_{1, R}(z)$ is the minimal eigenvalue of $H_R(z)$, then $R^2\lambda_{1, R}(z)=\lambda_1(R^{-1}z)$. Thus, changing variables we get \begin{eqnarray*}
		\E^{\varphi}(w) &=&\int_{\C^n}|\nabla^{(0,1)}w|^2e^{-2\varphi}+2\int_{\C^n}\lambda_1|w|^2 e^{-2\varphi}\\
		&=&\int_{\C^n} |\Dil_R\left(\nabla^{(0,1)}w\right)|^2e^{-2\varphi_R}+2\int_{\C^n}\lambda_1(R^{-1} z)|\Dil_Rw|^2 e^{-2\varphi_R}\\
		&=&R^2\left\{\int_{\C^n} |\nabla^{(0,1)}\left(\Dil_Rw\right)|^2e^{-2\varphi_R}+2\int_{\C^n}\lambda_{1,R}|\Dil_Rw|^2 e^{-2\varphi_R}\right\}.
	\end{eqnarray*} 
	This proves \eqref{eq:energy_scaling}. The rest of the statement follows by another simple change of variable.
\end{proof}

Condition \eqref{eq:spectral_gap_rescaled} is of course tailored for $\varphi$ homogeneous of degree $d$, that is, $\varphi(Rz)=R^d\varphi(z)$ for every $R>0$. In this case, in order to show that the $\dbar$-Neumann problem on the rigid domain associated to $\varphi$ satisfies a subelliptic estimate of order $s$ at the origin, one needs to show that	\begin{equation}\label{eq:sufficient_condition}
	\E^{\varphi}(w)\gtrsim R^{-2\delta}\int_{\C^n}|w|^2e^{-\varphi} \quad  \forall w\in C^\infty_c(\{|z|< R \}),\,\forall R \, \text{ large}, 
\end{equation} 
where $\delta=1-sd$. 

The remainder of this section consists of a series of standard manipulations of the subelliptic estimate \eqref{eq:subelliptic} that finally lead us to the quadratic forms $\mathbf{E}^{\xi\varphi}$ and the sufficient condition of Proposition \ref{prp:spectral_gap}. 

\subsection{Preliminaries on pseudoconvex rigid domains}\label{sec:pseudoconvex_rigid_domains} 

Since a defining function for $\Omega$ is given by $r(z,z_{n+1}) = \varphi(z)-\frac{z_{n+1}-\overline{z}_{n+1}}{2i}$, the $n$ $\C$-linearly independent vector fields \begin{equation}\label{eq:L_j}
	L_j = \frac{\partial}{\partial z_j} + 2i\frac{\partial \varphi}{\partial z_j}\frac{\partial}{\partial z_{n+1}}\quad (j=1,\ldots, n)
\end{equation}
restrict to a global frame of the CR bundle $T^{1,0}b\Omega:=\C \otimes Tb\Omega\cap T^{1,0}\C^{n+1}$. Adding $L_{n+1}=-2i\frac{\partial}{\partial z_{n+1}}$ to the collection \eqref{eq:L_j}, we obtain a global frame of $T^{1,0}\C^{n+1}$. We equip $\C^{n+1}$ with the Hermitian metric $h$ defined by the condition that this frame is orthonormal. The operators $\dbar$ and $\dbar^*$ will be defined with respect to this metric. In the sequel, we identify $b\Omega\subset \C^{n+1}$ with $\C^n\times \R$ via the global system of coordinates
\begin{equation}\label{eq:coordinates}
	z=(z_1,\ldots, z_n)\in\C^n, \quad t=\re(z_{n+1})\in \R.
\end{equation}

\subsection{A standard reduction to the boundary}\label{sec:reduction_boundary}
The boundary reduction to be discussed now is well-known and works equally well if the domain $\Omega$ is not rigid. We confine our discussion to the rigid case, in order to avoid introducing additional notation that is unnecessary for our purposes. 

If $\omega_1=dz_1,\ldots,\omega_n=dz_n, \omega_{n+1}=\partial r$ is the dual frame to $\{L_j\}_{1\leq j\leq n+1}$, a $(0,1)$-form $u$ can be uniquely represented as \[
u=\sum_{j=1}^{n+1}u_j\overline{\omega}_j
\]
and, if $u\in C^\infty_{c,\, (0,1)}(\C^{n+1})$, then $u_{|\Omega}\in \mathrm{dom}(\dbar^*)$ if and only if $u_{n+1}=0$ identically on $b\Omega$. We have the formulas \[
\dbar u=\sum_{1\leq j<k\leq n+1}\left(\overline L_ju_k-\overline L_ku_j\right)\, \overline\omega_j\wedge \overline\omega_k, \quad \dbar^*u=-\sum_{1\leq j \leq n+1}(L_j+b_j)u_j, 
\]
where $b_j=L_j(\log\rho)$ and $\rho$ is the density of the volume form of $h$ w.r.t.~Lebesgue measure. In view of the orthonormality of $\{\omega_j\}$, we have \[
|\dbar u|_h^2=\sum_{1\leq j<k\leq n+1}\left|\overline L_ju_k-\overline L_ku_j\right|^2.
\]
Hence, if $u$ is a smooth $(0,1)$-form supported on a fixed neighborhood of the origin and $u\in \mathrm{dom}(\dbar^*)$, then \begin{eqnarray*}
	Q(u)&:=&\lVert \dbar u\rVert^2_{L^2}+\lVert\dbar^*u\rVert^2_{L^2}+\lVert u\rVert^2_{L^2}\\
	&\simeq&\sum_{1\leq j<k\leq n+1}\int_\Omega\left|\overline L_ju_k-\overline L_ku_j\right|^2+\int_\Omega|\sum_{1\leq j\leq n+1}L_j u_j|^2+\sum_{1\leq j\leq n+1}\int_\Omega|u_j|^2
\end{eqnarray*}

As is well-known, the $L^2$-Sobolev norm of order $1$ of $u_{n+1}$ is controlled by $Q(u)$ (see, e.g., \cite[Lemma 2.12]{straube_book}). Thus, \begin{eqnarray}
	\notag Q(u)&\gtrsim& Q(u)+\lVert u_{n+1}\rVert_{W^1(\Omega)}^2\\
	\label{eq:Q_to_Qb}&\gtrsim& \sum_{1\leq j<k\leq n}\int_\Omega\left|\overline L_ju_k-\overline L_ku_j\right|^2+\int_\Omega|\sum_{1\leq j\leq n}L_j u_j|^2+\sum_{1\leq j\leq n}\int_\Omega|u_j|^2, 
\end{eqnarray}
where we are left only with tangential operators and components. Since there are no constraints on $u_1,\ldots, u_n$, our task can now be reduced to a boundary estimate. In detail, if $v=(v_1,\ldots, v_n)$ is an $n$-tuple of test functions defined on $b\Omega$, we set \[
Q_b(v)=\sum_{1\leq j<k\leq n}\int_{b\Omega}\left|\overline L_jv_k-\overline L_kv_j\right|^2+\int_{b\Omega}|\sum_{1\leq j\leq n}L_j v_j|^2+\sum_{1\leq j\leq n}\int_{b\Omega}|v_j|^2, 
\]
where integrals are with respect to Lebesgue measure in the boundary coordinates \eqref{eq:coordinates}. For later reference, we observe that \[
L_j=\frac{\partial}{\partial z_j}+i\frac{\partial \varphi}{\partial z_j}\frac{\partial}{\partial t}\qquad (j=1,\ldots, n), 
\] in these coordinates. Of course, translation invariance in the $t$-direction of these operators reflects translation invariance of the rigid domain $\Omega$. 

\begin{prp}[Reduction to a boundary subelliptic estimate]\label{prp:reduction_boundary}
	Assume that there exists a neighborhood $V$ of the origin on the boundary $b\Omega$ such that\begin{equation}\label{eq:boundary_subellipticity}
		Q_b(v)\gtrsim \lVert v\rVert_{W^s(b\Omega)}^2=\sum_{1\leq j\leq n}\lVert v_j\rVert_{W^s(b\Omega)}^2
	\end{equation} holds for every $n$-tuple $v=(v_1,\ldots, v_n)$ of test functions supported on $V$. Then the $\dbar$-Neumann problem on $\Omega$ satisfies a subelliptic estimate of order $s$ at the origin. 
\end{prp}

\begin{proof} Let $\widetilde{V}$ be a small neighborhood of the origin in $\C^{n+1}$ such that $\widetilde{V}\cap b\Omega\subseteq V$, and let $u\in C^\infty_{c, (0,1)}(\widetilde{V})$ be such that $u_{n+1}=0$ on $b\Omega$. \newline The frame $L_j$ is invariant under translations in $z_{n+1}$, and in particular in $\im(z_{n+1})$. Hence, applying \eqref{eq:boundary_subellipticity} to $v_h(z,t)=(u_j(z,t+i(\varphi(z))+h))_j$ and integrating in $h>0$, we conclude that \eqref{eq:Q_to_Qb} controls the tangential Sobolev norm of order $s$ of $u_j$ for $j=1,\ldots, n$. It is well-known that this implies the conclusion. See, e.g., Section 3 of \cite{kohn_1979}. 
\end{proof}

It is convenient to use the notation \begin{equation}\label{eq:E_b}
E_b(v):=\sum_{1\leq j<k\leq n}\int_{b\Omega}\left|\overline L_jv_k-\overline L_kv_j\right|^2+\int_{b\Omega}|\sum_{1\leq j\leq n}L_j v_j|^2, 
\end{equation}
for the "main term" in $Q_b$. 

\begin{rmk} The reader familiar with the tangential Cauchy--Riemann complex may recognize that $E_b$ is the quadratic form of the Kohn Laplacian $\Box_b$ on the CR hypersurface $b\Omega$, for an appropriate choice of metric and background measure.  \end{rmk}

\subsection{Two basic identities} We think of the $n$-tuples of test functions $v=(v_1,\ldots, v_n)$ of Proposition \ref{prp:reduction_boundary}  as $\C^n$-valued functions on $b\Omega\simeq \C^n_z\times \R_t$. In the next proposition, scalar differential operators act componentwise on $v$. Also, the $n\times n$ matrix-valued function $H(z)$ is lifted to a function on $b\Omega$ and acts on $\C^n$-valued functions by pointwise column-vector multiplication. 

\begin{prp}[Basic identities]\label{prp:basic_identities} Let $v\in C^\infty_c(b\Omega;\C^n)$. Then 
	\begin{eqnarray}
		\label{identity_+} E_b(v)&=&\sum_{1\leq j\leq n}\int_{b\Omega} |\overline{L}_jv|^2-2i\int_{b\Omega}\left( H\frac{\partial v}{\partial t},v\right)\\
		\label{identity_-} &=&\sum_{1\leq j\leq n}\int_{b\Omega} |L_jv|^2+2i\int_{b\Omega}\left((\mathrm{tr}(H)I_n-H)\frac{\partial v}{\partial t},v\right), 
	\end{eqnarray} 
	where $E_b$ is defined by formula \eqref{eq:E_b} and $I_n$ is the $n\times n$ identity matrix. 
\end{prp}

\begin{proof}
	This is a standard commutator computation. Expanding the squares in the formula defining $E_b(v)$ and integrating by parts twice one obtains \begin{eqnarray*}
		E_b(u)&=&\sum_{1\leq j,k\leq n}\left\{\int_{b\Omega} |\overline{L}_jv_k|^2 +\int_{b\Omega} (L_kv_k\overline{L_jv_j}-\overline{L}_jv_kL_k\overline{v}_j)\right\}\\
		&=&\sum_{1\leq j,k\leq n}\left\{\int_{b\Omega} |\overline{L}_jv_k|^2 +\int_{b\Omega} (L_kv_k\overline{L_jv_j}+\overline{L}_j\left(L_kv_k\right)\overline{v}_j+[L_k, \overline{L}_j]v_k\overline{v}_j)\right\}\\
		&=&\sum_{1\leq j,k\leq n}\left\{\int_{b\Omega} |\overline{L}_jv_k|^2 +\int_{b\Omega} [L_k, \overline{L}_j]v_k\overline{v}_j\right\}.
	\end{eqnarray*}
	Identity \eqref{identity_+} follows from the commutator formula $[L_k, \overline{L}_j]=-2i\frac{\partial^2\varphi}{\partial\overline{z}_j\partial z_k}\frac{\partial}{\partial t}$. Identity \eqref{identity_-} can then be deduced by the further integration by parts identity \[
	\int_{b\Omega}|\overline{L}_jv_k|^2 = \int_{b\Omega}(|L_j v_k|^2-[L_j,\overline{L}_j])v_k\overline{v}_k = \int_{b\Omega}\left(|L_j v_k|^2+2i\frac{\partial^2\varphi}{\partial\overline{z}_j\partial z_j}\frac{\partial v_k}{\partial t}\overline{v}_k\right). 
	\]
\end{proof}

\begin{rmk}\label{rmk:microlocal_ellipticity} Let $\delta>0$ be small. By assumption \eqref{eq:gradient=0}, in a small enough neighborhood of the origin $\left|\frac{\partial\varphi}{\partial z_j}\right|^2\leq \frac{\delta}{n}$ and the maximal eigenvalue of $H$ is bounded by a constant $C$. Then, if $v\in C^\infty_c(b\Omega; \C^n)$ is supported on such a neighborhood, we have 
	\[|\overline{L}_jv|^2\geq \frac{1}{2}\left|\frac{\partial v}{\partial \overline{z}_j}\right|^2-\frac{\delta}{n}\left|\frac{\partial v}{\partial t}\right|^2 \quad \text{and}\quad 2\left|\left( H\frac{\partial v}{\partial t},v\right)\right|\leq \delta\left|\frac{\partial v}{\partial t}\right|^2+\delta^{-1}C^2|v|^2.\]
	The basic identity \eqref{identity_+} then gives 
	\begin{eqnarray}
		\notag Q_b(v)&\gtrsim& \int_{b\Omega}\left\{\sum_{1\leq j\leq n}\left|\frac{\partial v}{\partial \overline{z}_j}\right|^2 - 2\delta\left|\frac{\partial v}{\partial t}\right|^2\right\}\\
		\label{microlocal_ellipticity} &\gtrsim& \int_{b\Omega}\left\{\frac{|\nabla_{\C^n} v|^2}{4} - 2\delta\left|\frac{\partial v}{\partial t}\right|^2\right\}, 
	\end{eqnarray}
	where $\nabla_{\C^n}$ is the ordinary gradient in $\C^n\equiv\R^{2n}$, and the last inequality is obtained by integration by parts. This elementary estimate expresses the microlocal ellipticity of $\Box_b$ in the complement of a conical neighborhood of the characteristic direction $\mathrm{span}\{\overline L_1,\ldots, \overline L_n\}^\perp\subset T^*b\Omega$. 
\end{rmk}

\subsection{Fourier analysis} 
If $f:b\Omega\rightarrow \C$ is a function, we denote by $\mathcal{F}f:\C^n\times \R\rightarrow \C$ its "vertical Fourier transform", that is, its Fourier transform in the $t$-variable, normalized as follows: \[
\mathcal{F}f(z,\xi) = \frac{1}{\sqrt{2\pi}}\int_{\R} f(z,t)e^{-i\xi t}dt \qquad (z\in \C^n,\, \xi\in \R). 
\] We will not comment on convergence of this and other integrals in what follows, as our functions $f$ will always be smooth and rapidly decaying in $t$. 

Applying Plancherel theorem in the variable $t$, we obtain the identity
\begin{equation}\label{eq:partial_Plancherel}
	\int_{b\Omega}f\overline{g}=\int_{\C^n\times \R} \mathcal{F}f\overline{\mathcal{F}g} 
\end{equation}
for $f,g:b\Omega\rightarrow \C$, where the implicit measures are Lebesgue in $(z,t)$ in the LHS and Lebesgue in $(z,\xi)$ in the RHS.

The vertical Fourier transform of a $\C^n$-valued function $v=(v_1,\ldots, v_n)$ is defined componentwise: $\mathcal{F}v(z,\xi):=(\mathcal{F}v_1(z,\xi), \ldots, \mathcal{F}v_n(z,\xi))$. 

We denote by $\mathcal{S}_{\mathrm{vert}}(b\Omega; \C^n)$ the space of smooth functions $v:b\Omega\rightarrow \C^n$ that are compactly supported in $z$ and Schwartz in $t$, the latter property meaning that $|t|^N\left|\frac{\partial^N v}{\partial t^N}\right|$ is bounded for every $N\in \N$. By standard facts of Fourier analysis, the vertical Fourier transform is an isomorphism of $\mathcal{S}_{\mathrm{vert}}(b\Omega; \C^n)$ with inverse $v\mapsto \mathcal{F}v(z,-\xi)$. 

If $w:\C^n\rightarrow \C^n$ is a function, we denote by $\nabla^{0,1}w$ (resp.~$\nabla^{1,0}w$) the $(0,1)$-part (resp.~the $(1,0)$-part) of its gradient, that is, the matrix $\nabla^{0,1}w=\left(\frac{\partial w_k}{\partial \overline z_j}\right)_{1\leq j,k\leq n}$ (resp.~$\nabla^{1,0}w=\left(\frac{\partial w_k}{\partial z_j}\right)_{1\leq j,k\leq n}$). The Hilbert--Schmidt norm of this matrix is denoted $|\nabla^{0,1}w|$ (resp.~$|\nabla^{1,0}w|$), that is, $|\nabla^{0,1}w|^2=\sum_{1\leq j,k\leq n}\left|\frac{\partial w_k}{\partial \overline z_j}\right|^2$ (resp.~$|\nabla^{1,0}w|^2=\sum_{1\leq j,k\leq n}\left|\frac{\partial w_k}{\partial z_j}\right|^2$). 

\begin{prp}[Fourier transformed basic identities]\label{prp:fourier}
	Given $v\in \mathcal{S}_{\mathrm{vert}}(b\Omega; \C^n)$, define the one-parameter families of $\C^n$-valued test functions \begin{eqnarray*}
		w_\xi^+(z)&=&e^{\xi\varphi(z)}\mathcal{F}v(z,\xi)\\
		w_\xi^-(z)&=&e^{\xi\varphi(z)}\mathcal{F}v(z,-\xi), 
	\end{eqnarray*}
	where $\xi\in \R$ and $z\in \C^n$. Then \begin{eqnarray*}
		E_b(v) &=& \int_\R\left\{\int_{\C^n} |\nabla^{0,1} w_\xi^+|^2e^{-2\xi\varphi}+2\xi\int_{\C^n}\left( Hw_\xi^+,w_\xi^+\right) e^{-2\xi\varphi}\right\}d\xi \\
		&=&\int_\R\left\{\int_{\C^n} | \nabla^{1,0}w_\xi^-|^2e^{-2\xi\varphi}+2\xi\int_{\C^n}\left((\mathrm{tr}(H)I_n-H)w_\xi^-,w_\xi^-\right) e^{-2\xi\varphi}\right\}d\xi. 
	\end{eqnarray*} 
\end{prp}

\begin{proof}
	The operator $-i\frac{\partial}{\partial t}$ and multiplication by $\xi$ are conjugated under $\mathcal{F}$, that is, \[
	\mathcal{F}\left(-i\frac{\partial f}{\partial t}\right)(z,\xi) = \xi\mathcal{F}f(z,\xi). 
	\]
	Hence, $L_j=\frac{\partial}{\partial z_j}+i\frac{\partial \varphi}{\partial z_j}\frac{\partial}{\partial t}$ is conjugated to \[\frac{\partial}{\partial z_j}-\xi\frac{\partial \varphi}{\partial z_j}=e^{\xi\varphi}\circ \frac{\partial}{\partial z_j}\circ e^{-\xi\varphi}.\] Similarly, $\overline{L}_j=\frac{\partial}{\partial \overline{z}_j}-i\frac{\partial \varphi}{\partial \overline{z}_j}\frac{\partial}{\partial t}$ is conjugated to \[\frac{\partial}{\partial \overline{z}_j}+\xi\frac{\partial \varphi}{\partial \overline{z}_j}=e^{-\xi\varphi}\circ \frac{\partial}{\partial \overline{z}_j}\circ e^{\xi\varphi}.\]
	
	By \eqref{identity_+} and Plancherel's identity \eqref{eq:partial_Plancherel},  
	\begin{eqnarray*}
		E_b(v)&=&\sum_{1\leq j\leq n}\int_{b\Omega} |\overline{L}_jv|^2-2i\int_{b\Omega} \left( H\frac{\partial v}{\partial t},v\right)\\
		&=&\int_\R\left\{\sum_{1\leq j\leq n}\int_{\C^n} \left|\frac{\partial(e^{\xi\varphi}\mathcal{F}v)}{\partial \overline{z}_j}\right|^2e^{-2\xi\varphi}+2\xi\int_{\C^n}\left( He^{\xi\varphi}\mathcal{F}v,e^{\xi\varphi}\mathcal{F}v\right) e^{-2\xi\varphi}\right\}d\xi.
	\end{eqnarray*} 
	This is the first stated identity. Similarly, we exploit the second basic identity \eqref{identity_-} to write	\begin{eqnarray*}
		E_b(v)&=&\sum_{1\leq j\leq n}\int_{b\Omega} |L_jv|^2+2i\int_{b\Omega} \left( (\mathrm{tr}(H)I_n-H)\frac{\partial v}{\partial t},v\right)\\
		&=&\int_\R\left\{\sum_{1\leq j\leq n}\int_{\C^n} \left|\frac{\partial(e^{-\xi\varphi}\mathcal{F}v)}{\partial z_j}\right|^2e^{2\xi\varphi}-2\xi\int_{\C^n}\left((\mathrm{tr}(H)I_n-H)e^{-\xi\varphi}\mathcal{F}v,e^{-\xi\varphi}\mathcal{F}v\right) e^{2\xi\varphi}\right\}d\xi.
	\end{eqnarray*} 
	Changing variables $\xi\leftrightarrow -\xi$, we obtain the second desired identity. 
\end{proof}

\subsection{Proof of Proposition \ref{prp:spectral_gap}}

Suppose that $v\in\mathcal{S}_{\mathrm{vert}}(b\Omega;\C^n)$ is vertically Fourier-supported on $\{\xi\geq 0\}$, that is, \[
\mathcal{F}v(z,\xi) = 0\quad \forall z\in\C^n \text{ and }\xi\leq 0. 
\]	
Then the first identity of Proposition \ref{prp:fourier} plus the trivial lower bound $H\geq \lambda_1I_n$ (as quadratic forms) gives \begin{eqnarray}
	\notag	E_b(v) &=& \int_0^{+\infty}\left\{\int_{\C^n} |\nabla^{0,1} w_\xi^+|^2e^{-2\xi\varphi}+2\xi\int_{\C^n}\left( Hw_\xi^+,w_\xi^+\right) e^{-2\xi\varphi}\right\}d\xi\\
	\label{eq:positive_support}	&\geq& \int_0^{+\infty}\left\{\int_{\C^n} |\nabla^{0,1} w_\xi^+|^2e^{-2\xi\varphi}+2\xi\int_{\C^n}\lambda_1|w_\xi^+|^2 e^{-2\xi\varphi}\right\}d\xi. 
\end{eqnarray}	

If instead $v\in\mathcal{S}_{\mathrm{vert}}(b\Omega;\C^n)$ is vertically Fourier-supported on $\{\xi\leq 0\}$, that is, \[
\mathcal{F}v(z,\xi) = 0\quad \forall z\in\C^n \text{ and }\xi\geq 0,  
\]	
then \begin{eqnarray}
	\notag	E_b(v) &=& \int_0^{+\infty}\left\{\int_{\C^n} |\nabla^{1,0} w_\xi^-|^2e^{-2\xi\varphi}+2\xi\int_{\C^n}\left((\mathrm{tr}(H)I_n-H)w_\xi^-,w_\xi^-\right) e^{-2\xi\varphi}\right\}d\xi\\
	\label{eq:negative_support}	&\geq& \int_0^{+\infty}\left\{\int_{\C^n} |\nabla^{1,0} w_\xi^-|^2e^{-2\xi\varphi}+2\xi\int_{\C^n}\lambda_1|w_\xi^-|^2 e^{-2\xi\varphi}\right\}d\xi, 
\end{eqnarray}
because the eigenvalues of the Hermitian matrix $\mathrm{tr}(H)I_n-H$ are \[\lambda_1+\ldots+\lambda_n-\lambda_j\qquad (1\leq j\leq n),\] and in particular they are all larger or equal to $\lambda_1+\ldots+\lambda_{n-1}\geq \lambda_1$.

Since $|\nabla^{1,0}w|=|\nabla^{0,1} \overline{w}|$, the lower bounds \eqref{eq:positive_support} and \eqref{eq:negative_support} can be rewritten as follows:  \begin{equation}\label{eq:positive_support_2}
	E_b(v) \geq \int_0^{+\infty}\E^{\xi \varphi}\left(w_\xi^+\right)d\xi\end{equation}
for $v\in \mathcal{S}_{\mathrm{vert}}(b\Omega; \C^n)$ positively Fourier-supported, and 
\begin{equation}\label{eq:negative_support_2}
	E_b(v) \geq \int_0^{+\infty}\E^{\xi \varphi}\left(\overline{w_\xi^-}\right)d\xi. 
\end{equation}
for $v\in \mathcal{S}_{\mathrm{vert}}(b\Omega; \C^n)$ negatively Fourier-supported. Recall that $w_\xi^+, w_\xi^-$ are defined as in Proposition \ref{prp:fourier} and that $\E^\varphi$ is defined on vector-valued functions by \eqref{eq:energy_form_vector}.
	
	By Proposition \ref{prp:reduction_boundary}, our goal is to prove that if $v=(v_1,\ldots, v_n)\in C^\infty_c(b\Omega; \C^n)$ is supported on a small neighborhood $V$ of the origin, then \begin{equation}\label{goal_subellipticity}
		Q_b(v) \gtrsim \lVert\Lambda^s v\rVert_{L^2}^2, 
	\end{equation}
	where $\Lambda^s$ is the standard pseudodifferential operator with symbol $(1+|\tau|^2)^{\frac{s}{2}}$. Here $\tau\in \R^{2n+1}$ is the dual, or cotangent, variable on $\C^n\times\R\simeq\R^{2n+1}$, and in particular $\tau_{2n+1}$ is dual to $t$. From now on, every norm will be an $L^2$ norm and we will omit subscripts. 
	
	Fix a test function $\chi_0(\tau)$ identically equal to one on a large ball centered at $0\in\R^{2n+1}$. Next, consider the covering of the unit sphere $\sfera^{2n} = \{\tau\in \R^{2n+1}\colon\ |\tau|=1\}$ given by the upper hemisphere $\Omega_+=\{\tau\in \sfera^{2n}\colon\ \tau_{2n+1}>0\}$, the lower hemisphere $\Omega_-\{\tau\in \sfera^{2n}\colon\ \tau_{2n+1}<0\}$, and the equatorial strip $\Omega_{\mathrm{eq}}=\{\tau\in \sfera^{2n}\colon\ |\tau_{2n+1}|<\delta\}$, where $\delta$ is a small positive parameter. We let $\{\sigma_+, \sigma_-, \sigma_{\mathrm{eq}}\}\subset C^\infty(\sfera^{2n})$ be a partition of unity subordinate to this covering. Finally, set \[
	\chi_\bullet(\tau) = (1-\chi_0(\tau))\sigma_\bullet\left(\frac{\tau}{|\tau|}\right) \qquad (\bullet\in \{\mathrm{eq}, +, -\}). 
	\]
	Denote by $P_\bullet$ the Fourier multiplier operator with symbol $\chi_\bullet(\tau)$. We let operators act componentwise on vector-valued functions. In order to prove inequality \eqref{goal_subellipticity}, it is enough to establish the four bounds \begin{equation}\label{microlocal_bounds}
		Q_b(v) \gtrsim \lVert\Lambda^s P_\bullet v\rVert^2\qquad (\bullet\in \{0, \mathrm{eq}, +,-\})
	\end{equation}
	for every $v\in C^\infty_c(V; \C^n)$. 
	
	The case $\bullet=0$ of \eqref{microlocal_bounds} is trivial, and the case $\bullet=\mathrm{eq}$ is a standard consequence of microlocal ellipticity in the equatorial conical region where $\chi_{\mathrm{eq}}$ is supported. See Remark \ref{rmk:microlocal_ellipticity} (recall that $s\leq 1$).
	
	We are left with the nonelliptic cases $\bullet=+$ and $\bullet=-$. This is where the key assumption \eqref{spectral_gap} is used. If $\overline{V}$ is contained in $\{|z|<C^{-1}\}\times \R$, and we choose a test function $\rho'\in C^\infty_c(\C^n)$ supported on $\{|z|<C^{-1}\}$ and identically one on $V$ (when thought of as a function on $b\Omega\simeq\C^n\times \R$), then 
	\[\lVert\Lambda^sP_+v\rVert^2\lesssim \left\lVert \left|\frac{\partial}{\partial t}\right|^sP_+v\right\rVert^2\lesssim\left\lVert \left|\frac{\partial}{\partial t}\right|^s\left(\rho'(z)P_+v\right)\right\rVert^2+\lVert v\rVert^2,\]
	where we used the fact that $(1+|\tau|^2)^{\frac{s}{2}}\lesssim |\tau_{2n+1}|^s$ on the support of $(1-\chi_0(\tau))\sigma_+\left(\frac{\tau}{|\tau|}\right)$ for the first bound, and the fact that $(1-\rho'(z))P_+$ is a smoothing operator for the second (recall that $v$ is supported on $V$). Here $\left|\frac{\partial}{\partial t}\right|^s$ denotes the operator with symbol $|\tau_{2n+1}|^s$. 
	
	Applying the partial Plancherel identity \eqref{eq:partial_Plancherel}, we find \begin{eqnarray}
		\notag \left\lVert \left|\frac{\partial}{\partial t}\right|^s\left(\rho'P_+v\right)\right\rVert^2&=&\int_{\R}|\xi|^{2s}\left\{\int_{\C^n}|\mathcal{F}\left(\rho'P_+v\right)|^2\right\}d\xi\\
		\label{partial_Plancherel2}&=&\int_{\R}|\xi|^{2s}\left\{\int_{\C^n}|\mathcal{F}\left(\rho'P_+v\right)e^{\xi\varphi}|^2e^{-2\xi\varphi}\right\}d\xi \end{eqnarray}
	Notice that since vertical Fourier transform and multiplication by $\rho'$ commute (because $\rho'$ is independent of $t$), the vertical Fourier support of $\rho'(z)P_+v(z,t)$ is contained in $\{\xi\geq C\}$ (if the ball where the cut-off function $\chi_0$ is identically one is chosen large enough) and $\mathcal{F}\left(\rho'P_+v\right)e^{\xi\varphi}$ is supported on $\{|z|<C^{-1}\}$ for every $\xi$. Hence, by \eqref{spectral_gap} and inequality \eqref{eq:positive_support_2} \begin{eqnarray*}
		\left\lVert \left|\frac{\partial}{\partial t}\right|^s\left(\rho'P_+u\right)\right\rVert^2&\lesssim& \int_C^{+\infty} \E^{\xi\varphi}\left(\mathcal{F}\left(\rho'P_+v\right)e^{\xi\varphi}\right)d\xi\\
		&\leq & E_b(\rho'P_+v)\\
		&\lesssim & Q_b(v), 
	\end{eqnarray*}
	where the last step is a standard commutation argument in pseudodifferential calculus. 
	This completes the proof of the microlocal bound \eqref{microlocal_bounds} with $\bullet=+$. The case $\bullet=-$ is entirely analogous. In particular, the reader should not worry about the conjugation appearing in \eqref{eq:negative_support_2}, since a valid replacement of \eqref{partial_Plancherel2} in the $\bullet=-$ case is \[
	\left|\left| \left|\frac{\partial}{\partial t}\right|^s\left(\rho'P_-u\right)\right|\right|^2 = \int_{C}^{+\infty}\xi^{2s}\left\{\int_{\C^n}|\overline{\mathcal{F}\left(\rho'P_-u\right)(z,-\xi)e^{\xi\varphi}}|^2e^{-2\xi\varphi}\right\}d\xi. 
	\]

\section{An uncertainty principle for the $\dbar$ operator}\label{sec:dbar_up}

The goal of this section is to prove a \emph{local lower bound} for the energy form $\E^\varphi$ \emph{in the one-dimensional case}, when $\Delta\varphi$ is comparable to the modulus squared of a complex polynomial. 

\begin{prp}[Main one-dimensional estimate]\label{lem:main_1d_estimate} 
	Let $P(\zeta)$ be a polynomial in one complex variable of degree $d$. Assume that all the roots of $P$ are in the disc $D\left(0,\frac{1}{2}\right)$ and that the leading coefficient of $P$ has modulus $A\geq C_d$, a positive constant depending only on the degree. 
	
	Let $\varphi\in C^2(\overline{\mathbb{D}}; \R)$ and $B>0$ be such that 
	\begin{equation}\label{main_1d_comparability}B^{-1}|P(\zeta)|^2\leq \Delta \varphi(\zeta)\leq B|P(\zeta)|^2 \qquad \forall\zeta\in \mathbb{D}.\end{equation}
	Then \begin{equation}\label{main_1d_estimate}
		\int_{\mathbb{D}}\left|\frac{\partial w}{\partial \overline{z}}\right|^2e^{-2\varphi}+\int_{\mathbb{D}}\Delta \varphi|w|^2e^{-2\varphi}\gtrsim_{d, B} A^{\frac{2}{d+1}}\int_{\mathbb{D}}|w|^2e^{-2\varphi}
	\end{equation}
	holds for all $w\in C^1(\overline{\mathbb{D}})$. 
\end{prp}

Inequality \eqref{main_1d_estimate} is sharp, as far as the dependence on $A$ is concerned. This may be seen considering $\varphi(\zeta)=A^2|\zeta|^{2d+2}$, in which case $\Delta\varphi(\zeta)\simeq_d |P|^2$, where $P(\zeta)=A\zeta^d$. The LHS of \eqref{main_1d_estimate}, evaluated on a test function $w$ supported on $\{\epsilon\leq |\zeta|\leq 2\epsilon\}$ and such that $\lVert w\rVert_{L^\infty}\lesssim 1$ and $\lVert \nabla w\rVert_{L^\infty}\lesssim \epsilon^{-1}$, is $O(1+A^2\epsilon^{2d+2})$, while $\int |w|^2e^{-2\varphi}\gtrsim \epsilon^2$, at least if $A^2\epsilon^{2d+2}\lesssim 1$. Optimising in $\epsilon$ in this range, we obtain $\frac{(\mathrm{LHS})}{\int |w|^2e^{-2\varphi}}\lesssim_d A^{\frac{2}{d+1}}$. \newline 
Notice that the hypothesis of the proposition is that $\Delta \varphi$ is comparable to the modulus squared of a perturbation of the polynomial $A\zeta^d$, so one should think of Proposition \ref{lem:main_1d_estimate} as a stability property of the "local spectral gap" of $\E^\varphi$. \newline 
We also remark that the proof for the special case $\Delta\varphi\simeq A^2|\zeta|^{2d}$ is significantly simpler, e.g., the clustering argument of Section \ref{sec:sublevel} is not needed. 

The proof of Proposition \ref{lem:main_1d_estimate} is given in Section \ref{sec:proof_main_one_dim}, after a number of preliminary facts are established. 

\subsection{A basic $\dbar$-uncertainty principle}
We begin with the following basic formulation of a $\dbar$-uncertainty principle ($\dbar$-UP), appearing in  \cite{dallara} as Lemma 12. We reproduce the proof (actually, a slight simplification thereof). 

\begin{lem}[Basic $\dbar$-UP]\label{lem:dbar_unc} Let $V:D(z,r)\rightarrow[0,+\infty)$ be a measurable function and define \[
	c:=\inf_{z'\in D(z,r)\setminus D\left(z,\frac{r}{2}\right)}V(z'). \]
	If $w\in C^1(\overline{D(z,r)})$, then \begin{equation}\label{dbar_unc}
		\int_{D(z,r)}\left|\frac{\partial w}{\partial \overline{z}}\right|^2+\int_{D(z,r)}V|w|^2\gtrsim \min\left\{c,\frac{1}{r^2}\right\}\int_{D(z,r)}|w|^2.
	\end{equation}
\end{lem}

\begin{proof} It is enough to prove \eqref{dbar_unc} in the case $z=0$ and $r=1$, that is, on the unit disc $\mathbb{D}$. The general case follows by scaling. Moreover, we may clearly assume that $\int_{\mathbb{D}}|w|^2=1$, and that \[
	V(z)=\begin{cases} 
		& 0 \qquad |z|<1/2\\
		& c \qquad 1/2\leq |z|<1
	\end{cases}.
	\]
	
	We use two inequalities. The first one is an elementary consequence of Cauchy integral formula:
	\begin{equation}\label{L2_annulus}
		\int_{\mathbb{D}}|h|^2\lesssim \int_{\mathbb{D}\setminus\frac{1}{2}\mathbb{D}}|h|^2 \qquad \forall h\in \mathcal{O}(\mathbb{D})\cap L^2(\mathbb{D}). 
	\end{equation}
	
	The second one is a Poincar\'e-type inequality for the operator $\frac{\partial}{\partial \overline{z}}$: \begin{equation}\label{dbar_poincare}
		\int_{\mathbb{D}}|w-Bw|^2\lesssim \int_{\mathbb{D}}\left|\frac{\partial w}{\partial \overline{z}}\right|^2\qquad \forall w\in C^1(\overline{\mathbb{D}}), 
	\end{equation}
	where $B$ is the Bergman projection of the unit disc ${\mathbb{D}}$, that is, the orthogonal projector of $L^2(\mathbb{D})$ onto the subspace of holomorphic functions $\mathcal{O}(\mathbb{D})\cap L^2(\mathbb{D})$. Inequality \eqref{dbar_poincare} is an immediate corollary of the $L^2$-solvability of the $\dbar$-equation on the unit disc (cf.~Theorem 4.3.4 of \cite{chen_shaw}, or Section 5 of \cite{dallara}). 
	
	We argue by contradiction, assuming that $w\in C^1(\overline{\mathbb{D}})$ satisfies $\int_{\mathbb{D}}|w|^2=1$ and 
	\[
	\int_{\mathbb{D}}\left|\frac{\partial w}{\partial \overline{z}}\right|^2+c\int_{\mathbb{D}\setminus \frac{1}{2}\mathbb{D}}|w|^2\leq a\min\left\{c,1\right\}, 
	\]
	where $a>0$ is a small absolute constant. It follows that $\int_{\mathbb{D}\setminus \frac{1}{2}\mathbb{D}}|w|^2\leq a$. Moreover, by \eqref{dbar_poincare}, we have \[\int_{\mathbb{D}}|w-Bw|^2\lesssim a.\]
	Using \eqref{L2_annulus}, we deduce that \begin{eqnarray*}
		\left(\int_{\mathbb{D}}|Bw|^2\right)^{\frac{1}{2}}&\lesssim&	\left(\int_{\mathbb{D}\setminus \frac{1}{2}\mathbb{D}}|Bw|^2\right)^{\frac{1}{2}}\\
		&\leq& \left(\int_{\mathbb{D}\setminus \frac{1}{2}\mathbb{D}}|w-Bw|^2\right)^{\frac{1}{2}} + \left(\int_{\mathbb{D}\setminus \frac{1}{2}\mathbb{D}}|w|^2\right)^{\frac{1}{2}} \\
		&\lesssim&\sqrt{a}
	\end{eqnarray*}
	For $a$ small enough, the inequalities $\lVert Bw\rVert_{L^2(\mathbb{D})}\lesssim \sqrt{a}$ and $\lVert w-Bw\rVert_{L^2(\mathbb{D})}\lesssim \sqrt{a}$ are in contradiction with our normalization $\lVert w\rVert_{L^2(\mathbb{D})}=1$. 
\end{proof}

\subsection{A weighted $\dbar$-UP}

What we actually need in our proof of Proposition \ref{lem:main_1d_estimate} is a generalization of the basic $\dbar$-UP of Lemma \ref{lem:dbar_unc}, where Lebesgue measure is multiplied by a factor $e^{-2\varphi}$, that is, we need lower bounds for the one-dimensional quadratic form  \begin{equation}\label{metric_hol_unc}
	\int_{D(z,r)}\left|\frac{\partial w}{\partial \overline{z}}\right|^2e^{-2\varphi}+\int_{D(z,r)}V|w|^2e^{-2\varphi},
\end{equation}
where $\varphi\in C^2(\overline{D(z,r)})$ satisfies appropriate assumptions.

\begin{rmk}\label{rmk:metric}
	One may interpret functions $w:D(z,r)\rightarrow \C$ as sections of a trivial holomorphic line bundle over the disc (the $\dbar$-operator being well-defined on sections of a holomorphic line bundle). Then the factor $e^{-2\varphi}$ can be interpreted as a metric on this bundle. Holomorphic changes of trivializations transform $\varphi$ into $\varphi + \re(G)$, where $G$ is a holomorphic function. Thus, the relevant properties of $\varphi$ are those that are invariant under such transformations, i.e., expressible in terms of the Laplacian $\Delta \varphi$, which expresses the curvature of the metric $e^{-2\varphi}$.
	
	Notice that the energy \eqref{metric_hol_unc} is stable under bounded perturbations of $\varphi$, and that we know how to bound it in the "flat" case $\varphi=0$ by Lemma \ref{lem:dbar_unc}. These two facts imply that an uncertainty principle holds whenever the metric has bounded curvature. This is the content of Lemma \ref{lem:dbar_unc_metric} below. 
\end{rmk}

\begin{lem}[Weighted $\dbar$-UP]\label{lem:dbar_unc_metric} Let $V:D(z,r)\rightarrow[0,+\infty)$ be a measurable function and define \[
	c:=\inf_{z'\in D(z,r)\setminus D\left(z,\frac{r}{2}\right)}V(z'). \]
	Let $\varphi\in C^2(\overline{D(z,r)}; \R)$ be such that $|\Delta \varphi(\zeta)|\leq Br^{-2}$ for all $\zeta\in D(z,r)$. Then \begin{equation}\label{dbar_unc_metric}
		\int_{D(z,r)}\left|\frac{\partial w}{\partial \overline{z}}\right|^2e^{-2\varphi}+\int_{D(z,r)}V|w|^2e^{-2\varphi}\gtrsim_B \min\left\{c,\frac{1}{r^2}\right\}\int_{D(z,r)}|w|^2e^{-2\varphi}
	\end{equation}
	holds for all $w\in C^1(\overline{D(z,r)})$. 
\end{lem}

\begin{proof} Under the hypothesis of the Lemma, $\varphi=\re\,  G+O(B)$, where $G$ is holomorphic. A proof of this known fact is provided in Appendix \ref{sec:appendix_A}. Applying Lemma \ref{lem:dbar_unc} to $w e^{-G}$, we get \[
	\int_{D(z,r)}\left|\frac{\partial w}{\partial \overline{z}}\right|^2e^{-2\re\, G}+\int_{D(z,r)}V|w|^2e^{-2\re\, G}\gtrsim \min\left\{c,\frac{1}{r^2}\right\}\int_{D(z,r)}|w|^2e^{-2\re\, G}, 
	\] which immediately yields the thesis. 
\end{proof}

\subsection{A Sublevel Set Lemma}\label{sec:sublevel}

In order to apply the weighted $\dbar$-UP of Lemma \ref{lem:dbar_unc_metric} to $V=|P|^2$, where $P$ is a complex polynomial, we need some quantitative control on the sublevel sets of $P$. This is the content of Lemma \ref{lem:sublevel} below, which in turn exploits the following elementary lemma. 

\begin{lem}[Clustering Lemma]
	Let $(X,d)$ be a finite metric space of cardinality $N$ and let $L>0$. There there exists a subset $S\subseteq X$ and a collection of radii $\{r_s\}_{s\in S}$ satisfying the following: \begin{enumerate}
		\item $L\leq r_s\leq 4^{N-1}L$, 
		\item if $s,s'\in S$ are distinct, then $d(s,s')> 2r_s+2r_{s'}$, 
		\item for every $x\in X$ there is a unique $s\in S$ such that $d(x,s)\leq r_s$. 
	\end{enumerate}
\end{lem}

\begin{proof}
	By the pigeonhole principle, for every $x\in X$ we can choose a radius $r_x\in [L, 4^{N-1}L]$ such that no point $s$ of $X$ satisfies the bounds $r_x< d(x,s)\leq 4r_x$. The set $S$ is built applying the following greedy algorithm: \begin{enumerate}
		\item Set initially $k=0$, $S_0:=$empty set and $X_0:=X$. Go to step (2). 
		\item If $X_k$ is empty, then halt. Otherwise go to step (3).
		\item Pick a point $s\in X_k$ with largest possible radius $r_s$ and define \[
		X_{k+1}:=X_k\setminus\{y\in X\colon\ d(y,s)\leq r_s\}, \quad S_{k+1}:=S_{k}\cup \{s\}. 
		\]
		\item Increase the counter $k$ by one and go back to step (2). 
	\end{enumerate}
	We omit the easy verification of the conditions in the statement. 
\end{proof}

\begin{lem}[Sublevel Set Lemma]\label{lem:sublevel}
	Let $P(z)$ be a nonzero polynomial of one complex variable of degree $d$. Denote by $A$ the absolute value of its leading coefficient. 
	
	Then there is a subset $R$ of the set of roots of $P$ and a collection of radii $\{r_\zeta\}_{\zeta\in R}$ with the following properties: \begin{enumerate}
		\item $r_\zeta\lesssim_d A^{-\frac{1}{d+1}}$, 
		\item on the disc $D(\zeta, 4r_\zeta)$, we have $|P(z)|\lesssim_d r_{\zeta}^{-1}$, 
		\item on the set $\C\setminus\left(\bigcup_{\zeta\in R}D(\zeta, 2r_\zeta)\right)$, we have $|P(z)|\gtrsim_d A^{\frac{1}{d+1}}$. 
	\end{enumerate}
\end{lem}

\begin{proof}
	We argue by induction on the degree, the case $d=0$ being trivial.
	
	If $P$ has degree $d\geq 1$ and leading coefficient of modulus $A$, we apply the Clustering Lemma to its set of roots with $L=A^{-\frac{1}{d+1}}$. We obtain a subset of roots $S$ and a collection of radii $r_s\simeq_dA^{-\frac{1}{d+1}}$ ($s\in S$). 
	
	Denote by $\mathcal{C}_s$ the "cluster at $s$", namely the set of roots of $P$ at distance at most $r_s$ from $s$. We also denote by $m_\zeta$ the multiplicity of $\zeta$ as a root of $P$ and we define $d_s:=\sum_{\zeta\in\mathcal{C}_s}m_\zeta$.
	
	We now argue differently depending on whether there is a single cluster or more than one. 
	
	\textbf{Case 1: there is only one cluster}. In this case $S=\{s\}$ and all the roots $\zeta$ of $P$ are in the disc $\overline{D(s, r_s)}$. Thus, on the disc $D(s,4r_s)$ we have \[
	|P(z)|=A\prod_{\zeta}|z-\zeta|^{m_\zeta}\lesssim_d A \left(A^{-\frac{1}{d+1}}\right)^d=A^{\frac{1}{d+1}}\simeq_d r_s^{-1}, \]
	and in the complement of the disc $D(s, 2r_s)$ we have the reverse inequality $|P(z)|\gtrsim_d A^{\frac{1}{d+1}}$. 
	
	\textbf{Case 2: there are two or more clusters}. Since all the roots are contained in $\bigcup_{s\in S}\overline{D(s, r_s)}$, on the complement of the set $\bigcup_{s\in S}D(s, 2r_s)$ we have the lower bound \begin{equation}\label{eq:case2}
	|P(z)|\geq A\prod_{s\in S} r_s^{d_s}\simeq_d A^{\frac{1}{d+1}}. 
	\end{equation}
	Let now $z\in D(s, 2r_s)$. Since $D(s, 2r_s)$ has empty intersection with $D(s', 2r_{s'})$ for every $s'\neq s$ and all the roots of $P$ in the latter disc (that is, those in the cluster $\mathcal{C}_{s'}$) are actually contained in $\overline{D(s',r_{s'})}$, it is easily seen that \[|z-\zeta|\simeq_d |s-s'| \qquad\forall z\in D(s,2r_s),\ \zeta\in \mathcal{C}_{s'}.\] 
	Hence, on $D(s,2r_s)$ we have \[
	|P(z)|\simeq_d A\prod_{s'\neq s} |s'-s|^{d_{s'}} \prod_{\zeta\in \mathcal{C}_s}|z-\zeta|^{m_\zeta}. 
	\]
	Let $P_s$ be the polynomial of degree $d_s<d$ with leading coefficient  \[A_s:=KA\prod_{s'\neq s} |s'-s|^{d_{s'}}\]
	that has a root of multiplicity $m_\zeta$ at every $\zeta\in \mathcal{C}_s$. Here $K$ is a positive constant depending only on $d$ to be fixed soon. By what we said above, $|P(z)|\simeq_d|P_s(z)|$ on the disc $D(s, 2r_s)$.

	Applying the inductive hypothesis to $P_s$, we obtain a subset $R_s\subseteq \mathcal{C}_s$ and a collection of radii $\{r_{s,\zeta}\}_{\zeta\in R_s}$. We now show that $R=\bigcup_{s\in S}R_s$ is the desired set of roots of $P$.
	
	Notice that the inequality $|s-s'|\gtrsim_d A^{-\frac{1}{d+1}}$ (valid for $s'\in S\setminus \{s\}$) gives \begin{equation}\label{rszeta}
		r_{s, \zeta}\lesssim_d A_s^{-\frac{1}{d_s+1}}\lesssim_d K^{-\frac{1}{d_s+1}}\left(A\prod_{s'\neq s}A^{-\frac{d_{s'}}{d+1}}\right)^{-\frac{1}{d_s+1}}=K^{-\frac{1}{d_s+1}}A^{-\frac{1}{d+1}}.
	\end{equation}
This verifies property (1) of the statement. Choosing $K$ large enough we may ensure that $r_{s, \zeta}\leq \frac{1}{4}r_s$ (recall that $r_s$ is comparable to $A^{-\frac{1}{d+1}}$, not only bounded above). This guarantees that, if $\zeta\in R_s$, then $D(\zeta, 4r_{s,\zeta})\subseteq D(s, 2r_s)$, so that the bound $|P_s(z)|\lesssim_d r_{s,\zeta}^{-1}$ valid on $D(\zeta, 4r_{s,\zeta})$ can be transferred to $|P|$. Property (2) is also verified.  \newline 
Finally, on the set \[
	D(s,2r_s)\setminus \left(\bigcup_{\zeta\in R_s}D(\zeta, 2r_{s,\zeta})\right)
	\] 
	we have \[|P(z)|\simeq_d|P_s(z)|\gtrsim_d A_s^{\frac{1}{d_s+1}}\gtrsim_d A^{\frac{1}{d+1}},\]
	where the last inequality is the same that we used in \eqref{rszeta}. Property (3) follows from \eqref{eq:case2} and the last inequality. 
\end{proof}

\subsection{Proof of Proposition \ref{lem:main_1d_estimate}}\label{sec:proof_main_one_dim}

	By the Sublevel Set Lemma (Lemma \ref{lem:sublevel}) we have a set $R\subseteq D(0,1/2)$ of cardinality at most $d$ and a collection of radii $r_\zeta\lesssim_dA^{-\frac{1}{d+1}}$, one for each $\zeta\in R$. Choosing $C_d$ large enough, we see that the discs $D(\zeta, 4r_\zeta)$ are contained in the unit disc. 
	
	By conditions (2) and (3) of the Sublevel Set Lemma, we are in a position to apply the weighted $\dbar$-UP (Lemma \ref{lem:dbar_unc_metric}), which yields \[
	\int_{D(\zeta, 4r_\zeta)}\left|\frac{\partial w}{\partial \overline{z}}\right|^2e^{-2\varphi}+\int_{D(\zeta, 4r_\zeta)}|P|^2|w|^2e^{-2\varphi}\gtrsim_{d, B} \min\{r_\zeta^{-2}, A^{\frac{2}{d+1}}\}\int_{D(\zeta, 4r_\zeta)}|w|^2e^{-2\varphi}.
	\]
	By condition (1), we see that the minimum above is $\simeq_d A^{\frac{2}{d+1}}$. Summing over $\zeta\in R$, we get \[
	\int_\Omega\left|\frac{\partial w}{\partial \overline{z}}\right|^2e^{-2\varphi}+\int_\Omega|P|^2|w|^2e^{-2\varphi}\gtrsim_{d, B} A^{\frac{2}{d+1}}\int_\Omega|w|^2e^{-2\varphi}, 
	\]
	where $\Omega=\bigcup_{\zeta\in R}D(\zeta, 2r_\zeta)$. By condition (3) of Lemma \ref{lem:sublevel} one gets a similar estimate in the complementary region $\mathbb{D}\setminus \Omega$. The proof is complete.

\section{Approximate minimal eigenvector fields}\label{sec:AME}

In this section we introduce the notion of approximate minimal eigenvector field for a plush function $\varphi$, and prove two basic propositions about it. 

\begin{dfn}[Approximate minimal eigenvector field]\label{dfn:ame}
	Let $\varphi$ be a smooth plush function defined near $p\in \C^n$. A germ of holomorphic vector field $X$ at $p$ is said to be an approximate minimal eigenvector field for $\varphi$ at $p$ if $X(p)\neq0$ and there is a positive constant $C$ such that  \begin{equation}\label{eq:ame}
	(H(z)X(z), X(z))\leq C \lambda_1(z) 
	\end{equation} for every $z$ in a neighborhood of $p$ (recall that $H$ is the Levi form of $\varphi$ defined in \eqref{eq:H}). 
\end{dfn}
Notice that the inequality $(H(z)X(z), X(z))\gtrsim \lambda_1(z) $ automatically holds near $p$, so \eqref{eq:ame} really expresses the fact that $X$ behaves approximately as an eigenvector field of minimal eigenvalue of the Levi form $H(z)$. 

The key role played by approximate minimal eigenvector fields in our proof of Theorem \ref{thm:main} stems from the following observation. Let $X$ be an approximate minimal eigenvector field for $\varphi$, and denote by $\exp_p(\zeta X)$ the solution of the complex ODE 	
\[\begin{cases}	&\frac{\partial}{\partial\zeta}\exp_p(\zeta X) = X(\exp_p(\zeta X))\\	&\exp_p(0 X)=p\end{cases}, \]
which is defined in a neighborhood of $0\in \C$. Then the mapping $\zeta\mapsto \exp_p(\zeta X)$ is holomorphic and \begin{eqnarray}
\notag	\Delta(\varphi(\exp_p(\zeta X))) &=& \left(H(\exp_p(\zeta X))\frac{\partial}{\partial\zeta}\exp_p(\zeta X), \frac{\partial}{\partial\zeta}\exp_p(\zeta X)\right)\\
\notag	&=& \left(H(\exp_p(\zeta X))X(\exp_p(\zeta X)), X(\exp_p(\zeta X))\right)\\
\label{eq:ame_observation}	&\simeq&\lambda_1(\exp_p(\zeta X)). 
\end{eqnarray}

Loosely speaking, along the analytic disc $\zeta\mapsto \exp_p(\zeta X)$, the plush function $\varphi$ (or more precisely, the metric $e^{-2\varphi}$) is as flat as possible (cf.~also Remark \ref{rmk:metric}), a very convenient property in light of the $\dbar$-UPs discussed in the previous section. 

The next proposition shows that approximate minimal eigenvector fields indeed exist, for an appropriate class of (germs of) plush functions. This is where we make a one-dimensionality assumption on the null-space of the Levi form. 

\begin{prp}[Existence of approximate minimal eigenvector fields]\label{prp:vector_field}
	Let $\varphi$ be a HSOnS defined near $p\in U$. Assume that the kernel of $H(p)$ is one-dimensional. Then there exists an approximate minimal eigenvector field for $\varphi$ at $p$. 
\end{prp}

In order to state our second proposition about approximate minimal eigenvector fields, we need to recall a notion of type introduced in \cite{catlin_necessary}. 

\begin{dfn}[Catlin's $1$-type]\label{dfn:catlin} We say that the collection of holomorphic maps \[\Psi=\{\psi_t:D(0,t)\rightarrow \C^{n+1}\}_{t\in (0,1]}, \]
	is a family of analytic discs shrinking to $q\in \C^{n+1}$ if \[
\lim_{t\rightarrow 0}\psi_t(0)=q,\quad	\sup_t\lVert\psi_t'\rVert_{L^\infty(D(0,t))}<+\infty,\quad \inf_t|\psi_t'(0)|>0. 
	\]
	
	If $q\in b\Omega$, then we define the order of contact $T(\Psi)$ of $\Psi=\{\psi_t\}_{t\in (0,1]}$ as the supremum over all $m\geq 0$ such that
	\[	\mathrm{dist}_{b\Omega}(\psi_t(\zeta))=O(t^m) \qquad \forall t\in (0,1],\quad \forall\zeta\in D(0,t).	\]
	Here $\mathrm{dist}_{b\Omega}$ is the Euclidean distance to the boundary (or any other comparable distance). 
	
	Let $\Omega\subset \C^{n+1}$ be a domain with smooth boundary. We define the Catlin's $1$-type of $b\Omega$ at $q$ as $T^1(b\Omega, q)=\sup T(\Psi)$, where the supremum is over all families of analytic discs shrinking to $q$.
\end{dfn}

\begin{rmk}
The definition above consists of an equivalent restatement of properties (i), (ii) and (iii) at p.~148 of \cite{catlin_necessary} in the case $q=1$ and $T=(0,1]$, plus the additional requirement that the centers of the analytic discs converge to a given point $q\in \C^{n+1}$. Notice that there is no loss in considering only $(0,1]$ as set of parameters for the family of discs. The terminology "family of analytic discs shrinking to $q$" is ours, as is the term "Catlin's $1$-type" for the resulting notion of type.  
\end{rmk}

Of course, $T^1(b\Omega, q)$ only depends on the germ of $b\Omega$ at $q$. 
 
\begin{prp}[Approx.~minimal eigenvector fields vs.~Catlin's $1$-type]\label{prp:type}
Let $\varphi:U\rightarrow\R$ be a smooth plush function, where $U\subseteq \C^n$ is open, and $\Omega=\{(z,z_{n+1})\in U\times \C\colon\, \im(z_{n+1})>\varphi(z))\}$ the associated pseudoconvex rigid domain. Let $p\in U$ and let $q=(p, \tau+i\varphi(p))$ be a boundary point of $\Omega$ "over $p$". Let $X$ be an approximate minimal eigenvector field for $\varphi$ at $p$ and $m_0\in \N$ be such that \[
\lambda_1(\exp_p(\zeta X))=O(|\zeta|^{m_0}).
\] Then \[
m_0+2\leq T^1(b\Omega, q). 
\]
	\end{prp}

The proof of Proposition \ref{prp:vector_field} is given in Section \ref{sec:holo_vf}, while the proof of Proposition \ref{prp:type} occupies Section \ref{sec:type}. 

\subsection{Existence of approximate minimal eigenvector fields}\label{sec:holo_vf}

Let $\varphi=|F|^2$, where \[F=(F_1,\ldots, F_n):U\rightarrow \C^n\] is a holomorphic mapping. Let $J$ be the complex Jacobian matrix of the mapping $F$, that is, \begin{equation}\label{eq:J}
	J_{jk}(z)=\frac{\partial F_j}{\partial z_k}(z)\qquad (1\leq j, k\leq n).
\end{equation} We will make use of the following elementary identity for the Levi form of $\varphi$ (defined in \eqref{eq:H}):\[
H_{jk}:=\frac{\partial^2\varphi}{\partial \overline{z}_j\partial z_k}=  \sum_{\ell=1}^n \overline{\frac{\partial F_\ell}{\partial z_j}}\frac{\partial F_\ell}{\partial z_k}, 
\]
that is, \begin{equation}\label{eq:J^*J}
	H= J^*J, 
\end{equation}
where $J^*$ is the (pointwise) conjugate transpose of $J$. Notice that $J$ is a square matrix and that $\det H(z)=|\det J(z)|^2$. An equivalent reformulation of \eqref{eq:J^*J} is \[
(H(z)v,v) = |J(z)v|^2\qquad \forall v\in \C^n,\,  \forall z \in U.
\]
Thus, $\ker H(z)=\ker J(z)$. 

	We need a basic integral formula from perturbation theory. Let $A$ be an $n\times n$ matrix with complex coefficients. Denote by $E(\lambda)$ the generalized eigenspace of $A$ of eigenvalue $\lambda$, that is, 
	\[
	E(\lambda) = \bigcup_{k\geq 1} \ker (\lambda I_n-A)^k, 
	\]
	where $I_n$ is the $n\times n$ identity matrix. Let $D(z,r)$ be a disc in the complex plane not containing any of the eigenvalues of $A$ in its boundary. Then \begin{equation}\label{matrix_residue}
		\Pi=\frac{1}{2\pi i}\int_{\partial D(z,r)} (\zeta I_n-A)^{-1}d\zeta, 
	\end{equation}
where the boundary circle is oriented counterclockwise, is a linear projector onto $\bigoplus_{\lambda\in D(z,r)}E(\lambda)$ (see, e.g., \cite{kato}[formula (1.16) at p. 67]). We want to apply formula \eqref{matrix_residue} to $A=J(z)$, in order to obtain a projection-valued holomorphic function of $z$. Notice that here we take advantage of the fact that $\varphi$ is a Hermitian Sum of $n$ Squares, that is, that $J(z)$ is a square matrix. 
	
	Since $J(p)$ has one-dimensional kernel, it has at most one Jordan block corresponding to the zero eigenvalue, which a priori could be of size $>1$. To avoid this, we use a simple trick. Let $M$ be a $n\times n$ unitary matrix such that the kernel and the range of $MJ(p)$ have trivial intersection. Such a unitary matrix exists, because the two subspaces have complementary dimensions and the modification $J(p)\mapsto MJ(p)$ rotates the range without affecting the kernel. As a result, $MJ(p)$ induces an invertible linear map on its range, and therefore the Jordan block of $MJ(p)$ corresponding to the zero eigenvalue must be of size $1$. Equivalently, the characteristic polynomial of $MJ(p)$ has a simple zero at $p$. 
	
Consider then the mapping $\widetilde{F}(z)=M(F(z))$, whose Jacobian is $\widetilde{J}(z)=MJ(z)$. By continuity and our choice of $M$, there is $\varepsilon>0$ and a neighborhood $V\subseteq U$ of $p$ such that, for every $z\in V$, the matrix $MJ(z)$ has only one eigenvalue $\alpha(z)$ of multiplicity $1$ such that $|\alpha(z)|<\varepsilon$, and no eigenvalue of modulus $=\varepsilon$. 
	
	By \eqref{matrix_residue}, there is a holomorphic matrix-valued map $\Pi(z)$ such that \begin{equation}\label{hol_projection}
		\widetilde{J}(z)\Pi(z) = \alpha(z)\Pi(z)\qquad \forall z\in V, 
	\end{equation}
	and $\Pi(p)$ is a linear projection onto $\ker \widetilde{J}(p) = \ker J(p)$. Let $v_0$ be a nonzero vector in the kernel of $J(p)$. Then the holomorphic vector field \[X(z):=\Pi(z)v_0\] satisfies $X(p)=v_0$ and \[
	\left(H(z)v(z), v(z)\right) = |J(z)X(z)|^2 = |\widetilde{J}(z)X(z)|^2 = |\alpha(z)|^2 |X(z)|^2, 
	\] for every $z\in V$. 
	Since every eigenvalue of $\widetilde{J}(z)$ different from $\alpha(z)$ has modulus $>\varepsilon$, we have \[
	|\det J(z)|^2 = |\det(\widetilde{J}(z))|^2\simeq_\varphi |\alpha(z)|^2\qquad \forall z\in V, 
	\] and since only one eigenvalue of $H$ vanishes at $0$, we similarly have 
	\[\det H(z)\simeq_\varphi \lambda_1(z).\] Both approximate identities hold throughout $V$ (shrinking it, if needed). Using $\det H(z)=|\det J(z)|^2$, we conclude that $\left(H(z)v(z), v(z)\right)\simeq_\varphi \lambda_1(z)|v(z)|^2$, as we wanted. 

\subsection{Approximate minimal eigenvector fields vs.~Catlin's $1$-type}\label{sec:type}

We are going to prove a statement slightly more general than Proposition \ref{prp:type}. To formulate it, we require a definition. 

\begin{dfn}\label{def:h_p}Let $\varphi:U\rightarrow\R$ be a smooth plush function, where $U$ is open and $p\in U$. We denote by $\mathfrak{h}_p(\varphi)$ the supremum over all $m\geq 0$ for which there exists a nonsingular analytic disc $\psi:\mathbb{D}\rightarrow U$ through $p$ (that is, $\psi(0)=p$ and $\psi'(0)\neq 0$) such that \begin{equation}\label{eq:flatness}
	\Delta(\varphi\circ \psi)(\zeta)=O(|\zeta|^m). 
	\end{equation}
\end{dfn}

From \eqref{eq:ame_observation}, one sees immediately that \begin{equation}\label{eq:VvsF}
m_0\leq \mathfrak{h}_p(\varphi), 
\end{equation}
where $m_0$ is as in the statement of Proposition \ref{prp:type}. 

The promised more general statement is the following. 

\begin{prp}[$\mathfrak{h}_p(\varphi)$ vs.~Catlin's $1$-type]\label{prp:FvsCatlin}
	Let $\varphi:U\rightarrow\R$ be a smooth plush function and $\Omega=\{(z,z_{n+1})\in U\times \C\colon\, \im(z_{n+1})>\varphi(z))\}$ the associated pseudoconvex rigid domain. Let $p\in U$ and let $q=(p, \tau+i\varphi(p))$ be a boundary point of $\Omega$ "over $p$". Then \[
    \mathfrak{h}_p(\varphi)+2\leq T^1(b\Omega, q). 
	\]
\end{prp}

\begin{proof}
	By translation invariance in $\re(z_{n+1})$, we may assume that $\tau=0$. Let $\psi:\mathbb{D}\rightarrow U$ be a nonsingular analytic disc through $p$ such that \eqref{eq:flatness} holds. Rescaling the $\zeta$-variable, we may assume that $|\psi'(\zeta)|\simeq 1$ on the whole disc. Consider the subharmonic function $f(\zeta)=\varphi(\psi(\zeta))-\varphi(p)$. We have \[\lVert f\rVert_{L^\infty(D(0,t))}\lesssim t,\qquad \lVert\Delta f\rVert_{L^\infty(D(0,t))}\lesssim t^{m+2}t^{-2}\qquad \forall t\in (0,1].\]
	Hence, Lemma \ref{lem:well_known} gives a holomorphic function $G_t\in \mathcal{O}(D(0,t))$ such that \[
	\varphi\circ \psi=\im(iG_t+i\varphi(p))+O(t^{m+2})\quad \text{on }D(0,t)\] and \[
	t\lVert G_t'(\zeta)\rVert_{L^\infty(D(0,0.5t)}\lesssim t+t^{m+2}\lesssim t.
	\] 
	Subtracting an imaginary constant from $G_t$, we may assume that $\im(G_t(0))=0$. 
	Define\[
	\psi_t(\zeta):=(\psi(0.5\zeta), iG_t(0.5\zeta)+i\varphi(p))\qquad (\zeta\in D(0,t))
	\]
	By the estimates just established, one sees that $\Psi=\{\psi_t\}_{t\in (0,1]}$ is a family of analytic discs shrinking to $q$ (notice that $\im(iG_t(0))=O(t^{m+2})$, so $\psi_t(0)$ converges to $q=(p, i\varphi(p))$ as $t$ goes to zero). Finally, evaluating the defining function $r(z,z_{n+1})=\varphi(z)-\im(z_{n+1})$ on the discs we see that the family $\Psi$ has order of contact at least $m+2$ with the boundary. Hence $T^1(b\Omega; p+i\varphi(p))\geq m+2$. The conclusion follows taking the supremum over all nonsingular analytic discs through $p$. 
\end{proof}

\section{Sharp subelliptic estimates for a class of dilation-invariant special domains}\label{sec:ker_one_dim}

In this section we show how the tools introduced above allow to exactly determine the sharp order of subellipticity for an interesting class of domains. We begin with the definition of the relevant class of plush functions. 

\begin{dfn}\label{def:homog_special} Let $n$ and $d$ be positive integers. We define $\mathcal{H}_{n,d}$ as the class of plush functions $\varphi:\C^n\rightarrow \R$ satisfying the following properties: \begin{enumerate}
\item $\varphi$ is a HSOnS. 
\item $\varphi$ is $\C$-homogeneous of degree $2d$, that is, \[
\varphi(\lambda z)=|\lambda|^{2d}\varphi(z)\qquad \forall z\in \C^n,\quad \forall \lambda\in \C^\times, 
\] where $d\geq 1$ is an integer. 
\item $\varphi$ vanishes at the origin and only there. 
	\end{enumerate}
\end{dfn}
By condition (1), $\varphi=\sum_{\ell=1}^n|F_\ell|^2$, with each $F_\ell:\C^n\rightarrow \C$ holomorphic, while condition (2) is equivalent to saying that each $F_\ell$ is a polynomial of degree $d$. 

Condition (3) ensures that the mapping $F=(F_1,\ldots, F_n)$ descends to a non-constant holomorphic map (or a regular morphism, in the language of algebraic geometry) \[\widetilde{F}:\mathbb{P}^{n-1}\rightarrow \mathbb{P}^{n-1}.\] 

By condition (3), $0\in \C^n$ is an isolated point of the variety $\{F_\ell(z)=0\quad \forall \ell\}$, which is known to be equivalent to the finiteness of the D'Angelo type $\Delta^1(b\Omega; 0)$, where \[
\Omega=\{(z,z_{n+1})\in \C^{n+1}\colon\, \im(z_{n+1})>\varphi(z))\}
\] is the rigid domain associated to $\varphi$. See \cite[Proposition 4.3]{catlin_dangelo}. We call domains of this form \emph{dilation-invariant special domains}. 

The next definition singles out the key quantity associated to the self-mapping $\widetilde{F}$ of $\mathbb{P}^{n-1}$ that dictates which subelliptic estimates at the origin hold on $\Omega$. We recall that the order at $0$ of a holomorphic map $f:\mathbb{D}\rightarrow \C^n$ (or a germ thereof) is the quantity \[
\nu_0(f) = \min\{k\geq 1\colon\, f^{(k)}(0)\neq 0\}
\]
if $f$ is nonconstant, while $\nu_0(f)=+\infty$ if $f$ is constant. Equivalently, $\nu_0(f)=\sup\{m\geq 0\colon\, |f(\zeta)-f(0)|=O(|\zeta|^m)\}$, from which it is apparent that the order is well-defined for every holomorphic $f:\mathbb{D}\rightarrow Y$, where $Y$ is a complex manifold. 

\begin{dfn}\label{def:flatness} If $G:\mathbb{P}^{n-1}\rightarrow \mathbb{P}^{n-1}$ is holomorphic, we denote by $\mathfrak{t}(G)$ the maximal order of flatness of $G$ along nonsingular analytic discs, that is, \[
\mathfrak{t}(G)=\sup_{\psi} \nu_0(G\circ \psi), 
\]
where $\psi: \mathbb{D}\rightarrow \mathbb{P}^{n-1}$ is holomorphic and satisfies $\psi'(0)\neq 0$. 
\end{dfn}

We can finally state our sharp subellipticity theorem. 

\begin{thm}[Sharp subelliptic estimates]\label{thm:main}
	
	Let $\varphi=|F|^2$ be a plush function in the class $\mathcal{H}_{n,d}$, $H$ its Levi form (see \eqref{eq:H}), and $\Omega$ the associated dilation-invariant special domain. Assume moreover that \begin{equation}\label{eq:one_dim_kernel}\dim \ker H(p)\leq 1 \qquad \forall p\in \C^n\setminus\{0\}.\end{equation} 
	
	Then the $\dbar$-Neumann problem on $\Omega$ satisfies a subelliptic estimate at the origin of order $s=(2\max\{d, \mathfrak{t}(\widetilde{F})\})^{-1}$ and no better one. 
	
	The quantity $2\max\{d, \mathfrak{t}(\widetilde{F})\}$ equals the upper semicontinuous envelope of Catlin's type $\widetilde{T}^1(b\Omega, 0)$ (see Definition \ref{dfn:catlin} and \eqref{eq:catlin_type} below). Hence, \[
	s(\Omega; 0)=\frac{1}{\widetilde{T}^1(b\Omega, 0)}. 
	\] 
\end{thm}

Before passing to the proof of the theorem, we give a couple of remarks that may clarify the scope of its statement. \begin{enumerate}
	\item Let $\pi:\C^n\setminus\{0\}\rightarrow  \mathbb{P}^{n-1}$ be the canonical projection. It is easy to see that the differentials of the mappings $F:\C^n\rightarrow \C^n$ and $\widetilde{F}:\mathbb{P}^{n-1}\rightarrow \mathbb{P}^{n-1}$ at $p$ and $\pi(p)$ respectively have kernels of the same dimension. Thus, hypothesis \eqref{eq:one_dim_kernel} in Theorem \ref{thm:main} can be reformulated as \begin{equation}\label{eq:one_dim_ker_bis}
		\dim \ker d\widetilde{F}(q)\leq 1\qquad \forall q\in \mathbb{P}^{n-1}. 
	\end{equation}
	\item If $n=1$ or $n=2$, then the hypothesis \eqref{eq:one_dim_kernel} is automatically satisfied and $\mathfrak{t}(\widetilde{F})\leq d$.
	\item If $n=3$ or $n=4$, then the hypothesis is generically satisfied in every degree. The invariant $\mathfrak{t}(\widetilde{F})$ may exceed the degree $d$, leading to a more complex behavior of the invariant $s(\Omega, 0)$ on dilation-invariant special domains. 
	\item If $n\geq 5$, no $\varphi\in \mathcal{H}_{n,d}$ satisfies \eqref{eq:one_dim_kernel} at every point $p\neq 0$, that is, there is no domain in dimension $6$ or higher to which our theorem applies. 
\end{enumerate}
Proofs of these facts are in Section \ref{sec:examples}. We prove Theorem \ref{thm:main} by reducing it to a key estimate (Section \ref{sec:key_estimate}), whose proof is given in Sections \ref{sec:local_bound} through \ref{sec:scaling}. 

\subsection{Reduction to a key estimate}\label{sec:key_estimate}

Recall the following fundamental theorem of Catlin, which we formulate using the terminology of Definition \ref{dfn:catlin}. 

\begin{thm}[Catlin's necessary condition for subellipticity, \cite{catlin_necessary}]\label{thm:catlin_necessity}
	Let $\Omega\subset\C^{n+1}$ be smooth and pseudoconvex. Assume that $q\in b\Omega$ is a boundary point and that a subelliptic estimate of order $s$ holds at $q$. If $\Psi$ is a family of analytic discs shrinking to $q$ with order of contact $T(\Psi)$, then $s\leq \frac{1}{T(\Psi)}$. Hence, \[
	s(\Omega, q)\leq \frac{1}{T^1(b\Omega, q)}. 
	\]
\end{thm}

Since $s(\Omega, q)$ is lower semicontinuous in $q$, we are allowed to replace $T^1(b\Omega, q)$ in the above estimate with its upper semicontinuous envelope $\widetilde{T}^1(b\Omega, q)$, i.e., the smallest upper semicontinuous function larger than $T^1(b\Omega, q)$:  \begin{equation}\label{eq:catlin_type}
	\widetilde{T}^1(b\Omega, q)=\lim_{\epsilon\rightarrow 0}\sup_{|q'-q|<\epsilon} T^1(b\Omega, q').
\end{equation}
Indeed, one may directly see that the resulting inequality \[
s(\Omega, q)\leq \frac{1}{\widetilde{T}^1(b\Omega, q)}
\]
follows from Theorem 3 of \cite{catlin_necessary}. 

The following proposition compares the "regularized" Catlin's $1$-type $\widetilde{T}^1(b\Omega, 0)$ with the invariants of Definition \ref{def:h_p} and Definition \ref{def:flatness}, when $\Omega$ is a dilation-invariant special domain. 

\begin{prp}[$\mathfrak{h}_p(\varphi)$ vs.~Catlin's $1$-type on dilation-invariant special domains]\label{prp:type_comparison}
	Let $\varphi\in \mathcal{H}_{n,d}$. We have
\[
	2\max\{d, \mathfrak{t}(\widetilde{F})\}=\sup_{p\in \C^n}\mathfrak{h}_p(\varphi)+2\leq \widetilde{T}^1(b\Omega, 0).
\]
\end{prp}

\begin{proof}
	
By homogeneity, $\sup_{p\in B(0,\epsilon)} \mathfrak{h}_p(\varphi) = \sup_{p\in \C^n}\mathfrak{h}_p(\varphi)$ for every $\epsilon>0$. Thus, the inequality follows from Proposition \ref{prp:FvsCatlin}. 

If $\psi$ is a nonsingular analytic disc in $\C^n$, then \[
\Delta(\varphi\circ \psi) = \sum_{\ell=1}^n\left|\frac{\partial (F_\ell\circ \psi)}{\partial \zeta}\right|^2.
\]
It follows that the largest $m$ such that $\Delta(\varphi\circ \psi)=O(|\zeta|^m)$ equals $2\nu_0(F\circ \psi)-2$. If the disc is through the origin, then $\nu_0(F\circ \psi)=d$. If $\psi'(0)$ is parallel to $\psi(0)$, then $\nu_0(F\circ \psi)=1$. Thus, it is enough to show that \[
\sup_{\psi}\nu_0(F\circ \psi) = \mathfrak{t}(\widetilde{F}), 
\]
where the supremum is over all analytic discs $\psi$ such that $\psi(0)\neq 0$ and $\psi'(0)\notin \C\psi(0)$. Denoting $\pi:\C^n\setminus \{0\}\rightarrow \mathbb{P}^{n-1}$ the canonical projection, it is clear that $\nu_0(F\circ \psi)\leq \nu_0(\pi\circ F\circ \psi)=\nu_0(\widetilde{F}\circ \widetilde{\psi})$, where $\widetilde\psi = \pi \circ \psi$. On the other hand, we claim that every nonsingular analytic disc $\widetilde{\psi}$ in $\mathbb{P}^{n-1}$ has a lift $\psi$ in $\C^n$ such that $\nu_0(F\circ \psi)=\nu_0(\widetilde{F}\circ \widetilde{\psi})$. To verify the claim, one may assume without loss of generality that $\psi(0)=(0,\ldots, 0,1)$ and $F(\psi(0,\ldots, 0,1))=(0,\ldots, 0,1)$ and work in local coordinates \[
(w,\zeta):=(w_1,\ldots, w_{n-1}, \zeta)=(z_1z_n^{-1}, \ldots, z_{n-1}z_n^{-1}, z_n)\]
near that point, so that $\pi$ becomes the projection onto the first $n-1$ components. The mapping $F$ takes the form $(w,\zeta)\mapsto (Q(w), \zeta^dR(w))$, where $Q$ and $R$ are holomorphic. If the nonsingular analytic disc $\widetilde\psi$ is expressed as $w(t)$ in coordinates, then it is enough to set $\psi(t)=(w(t), R(w(t))^{-\frac{1}{d}})$. 
\end{proof}

We now state the key inequality needed for the proof of Theorem \ref{thm:main}. 

\begin{lem}[Key Lemma]\label{lem:main_estimate}
	Let $\varphi\in \mathcal{H}_{n,d}$. 
	\begin{enumerate}
\item 	 For every $C$, there exists $\epsilon>0$ such that \begin{equation}\label{unit_ball_homog}
		\E^\varphi(w)\geq\epsilon \int_{\{|z|\leq C\}} |w|^2e^{-2\varphi}\qquad\forall w\in C^\infty_c(\C^n). 
	\end{equation}
	
\item 	Let $p\in \C^n\setminus\{0\}$ be such that  $\dim \ker H(p)\leq 1$. Then there exists $C<+\infty$ and an $\R^+$-invariant neighborhood $V$ of $p$ such that   \begin{equation}\label{eq:main_estimate}
	\E^\varphi(w)\geq C^{-1} \int_{V\cap \{|z|\geq C\}} |z|^{2\left(\frac{2d}{\mathfrak{h}_p(\varphi)+2}-1\right)}\, |w|^2e^{-2\varphi}
	\end{equation} for every $w\in C^\infty_c(\C^n)$.
\end{enumerate}
\end{lem}

Before delving into the proof of Lemma \ref{lem:main_estimate}, let us show how it implies Theorem \ref{thm:main}. 

\begin{proof}[Proof of Theorem \ref{thm:main}]
By part (2) of Lemma \ref{lem:main_estimate}, every $p\in \sfera^{2n-1}$ admits an $\R^+$-invariant open neighborhood $V_p$ and a constant $C_p>0$ such that \[
	\E^\varphi(w)\geq C_p^{-1} \int_{V_p\cap \{|z|\geq C_p\}} |z|^{2\left(\frac{2d}{\mathfrak{h}_p(\varphi)+2}-1\right)}\, |w|^2e^{-2\varphi}\qquad \forall w\in C^\infty_c(\C^n). 
	\]
	A compactness argument gives then a constant $C$ such that \[
	\E^\varphi(w)\geq C^{-1} \int_{|z|\geq C} |z|^{-2\left(1-2ds\right)} \, |w|^2e^{-2\varphi}\qquad \forall w\in C^\infty_c(\C^n), 
	\]
	where $s=\frac{1}{\sup_{p}\mathfrak{h}_p(\varphi)+2}$. Combining this estimate and part (1) of Lemma \ref{lem:main_estimate}, we see that the sufficient condition \eqref{eq:sufficient_condition} with this value of $s$ is satisfied. Thus, \[
	s(\Omega; 0)\geq \frac{1}{\sup_{p}\mathfrak{h}_p(\varphi)+2}. 
	\] 
By Catlin's necessary condition and Proposition \ref{prp:FvsCatlin}, we have the chain of estimates \[
	\widetilde{T}^1(b\Omega, 0) \leq \frac{1}{s(\Omega; 0)}\leq \sup_{p\in \C^n}\mathfrak{h}_p(\varphi)+2=2\max\{d, \mathfrak{t}(\widetilde{F})\} \leq \widetilde{T}^1(b\Omega, 0). 
	\]
	
	The proof of Theorem \ref{thm:main} is complete. 
	\end{proof}

We now turn to the proof of Lemma \ref{lem:main_estimate}. 

\subsection{Proof of the local energy bound}\label{sec:local_bound}
 
We now prove item (1) in Lemma \ref{lem:main_estimate}, which is a simple application of the basic $\dbar$-UP (Lemma \ref{lem:dbar_unc}). 

Pick $p\in \sfera^{2n-1}$ such that $\lambda_1(p)>0$. By continuity and $\C$-homogeneity there exists a small constant $\epsilon>0$ such that $\lambda_1$ does not vanish on \[
D':=\{z\in \C^n\colon\ 1/2\leq |(z,p)|\leq 1,\, |z-(z,p)p|\leq \epsilon\}.
\]  Consider the set \[
D:=\{z\in \C^n\colon\ |(z,p)|\leq 1,\, |z-(z,p)p|\leq \epsilon\}.
\]
Cutting $D$ into discs parallel to $p$ and applying the basic $\dbar$-UP to each disc, one gets \begin{eqnarray*}
	\int_D|\nabla^{0,1}w|^2+\int_D\lambda_1 |w|^2\gtrsim \int_D|w|^2\geq \int_{|z|\leq \epsilon}|w|^2.
\end{eqnarray*}
Applying this estimate to $w(\epsilon^{-1} Cz)$, exploting the homogeneity of $\lambda_1$ one finds that \begin{eqnarray*}
	\int_{\epsilon^{-1}CD}|\nabla^{0,1}w|^2+\int_{\epsilon^{-1}CD}\lambda_1 |w|^2\gtrsim_C \int_{|z|\leq C}|w|^2.
\end{eqnarray*}
Since $e^{-2\varphi}\simeq_C 1$ on $\epsilon^{-1}CD$ and $\{|z|\leq C\}$, we finally obtain \eqref{unit_ball_homog}. 

Notice that all we used about $\varphi$ for the local energy bound is its $\C$-homogeneity.

Let us move to the proof of part (2) of Lemma \ref{lem:main_estimate}.  
 
\subsection{Preparation}\label{sec:preparation}

The case $\dim \ker H(p)= 0$ is trivial. In fact, $\mathfrak{h}_p(\varphi)=0$ in this case and the conclusion follows from the fact that $\lambda_1(z)\gtrsim |z|^{2d-2}$ on an $\R^+$-invariant neighborhood of $p$. From now on, we assume that $\dim\ker H(p)=1$. 

By Proposition \ref{prp:vector_field}, there exists an approximate minimal eigenvector $X$ defined on an open neighborhood $W$ of $p$. 

Let $B$ be a small $\C$-codimension one ball perpendicular to $X(p)$, that is, 
\[
B=\{p+a\colon\ a\perp X(p), \, |a|<\varepsilon\}.
\]
Denote by $\Theta(z';\zeta)$ the solution of the complex ODE \[
\begin{cases}
	&\frac{\partial\Theta}{\partial\zeta}(z';\zeta) = X(\Theta(z';\zeta))\\
	&\Theta(z';0)=z'\in B
\end{cases}.
\]  If $B$ and $\tau>0$ are small enough, we obtain a holomorphic mapping \[\Theta: B\times D(0,\tau)\rightarrow W.\]
The transversality of $\frac{\partial\Theta}{\partial\zeta}(p;0)=X(p)$ and $B$ ensures, shrinking $B$ and $\tau$ enough, that $\Theta$ induces a biholomorphism onto its (open) image. We also shrink $W$ so that $\Theta$ itself is a biholomorphism. In other words, what we did is foliating a neighborhood of $p$ with analytic discs whose complex tangent vectors are approximate minimal eigenvectors of $H$. 

We now consider the restriction of the plurisubharmonic function $\varphi$ to these analytic discs. The same computation as in \eqref{eq:ame_observation} gives \[
\Delta_\zeta (\varphi\circ \Theta(z'; \zeta)) \simeq\lambda_1(\Theta(z'; \zeta)).
\]
Since $\dim\ker H(p)=1$, we have $\lambda_1(z)\simeq \det H(z)=|\det J(z)|^2$ near $p$. Hence, $\lambda_1(\Theta(z'; \zeta))$ is comparable to the modulus squared of the holomorphic function \[\mathcal{J}(z'; \zeta):=\det J(\Theta(z'; \zeta)).\] Notice that $\mathcal{J}(p;0)=0$. By the Weierstrass Preparation Theorem we have two cases to consider: \begin{itemize}
	\item[(a)] $\zeta\mapsto \mathcal{J}(p;\zeta)$ vanishes identically on $D(0,\tau)$; 
	\item[(b)] $\mathcal{J}(z';\zeta)$ equals, modulo a unit of the ring of germs of holomorphic functions at $(p;0)$, a Weierstrass polynomial \[
	\zeta^m+\sum_{k=1}^m c_k(z')\zeta^{m-k}, 
	\]
	where each $c_j(z')$ is a holomorphic germ vanishing at $p$. 
	\end{itemize}

In case (a), $\ker J$ is nontrivial on the whole range of the analytic disc $\zeta\mapsto \Theta(p;\zeta)$. Since $\frac{\partial\Theta}{\partial\zeta}(z';\zeta)$ lies in the kernel of $J(\Theta(z';\zeta))$ for every point at which the latter is nontrivial, the chain rule shows that $F\circ  \Theta(p;\zeta)$ is independent of $\zeta$. (Viceversa, it is easily seen that if $F$ is constant on the disc through $p$ then necessarily $\mathcal{J}(p;\zeta)$ vanishes identically in $\zeta$.) Thus, case (a) above may be ruled out because $F$ has finite fibers, as a consequence of Definition \ref{def:homog_special}.  

Notice that the degree $m\geq 1$ of the Weierstrass polynomial clearly satisfies (cf.~Definition \ref{dfn:ame}) $2m \leq \mathfrak{h}_p(\varphi)$ (recall \eqref{eq:VvsF}). 

Let us record in a proposition what we obtained so far. 

\begin{prp}[Convenient foliation of a neighborhood of $p$]\label{prp:preparation}
Under the assumptions of Lemma \ref{lem:main_estimate} there is a biholomorphism \[
\Theta: B^{n-1}\times \mathbb{D}\rightarrow W, 
\]
where $B^{n-1}\subseteq \C^{n-1}$ is the unit ball, such that $\Theta(0;0)=p$ and\begin{equation}\label{eq:preparation}
	\Delta_\zeta (\varphi\circ \Theta(z'; \zeta)) \simeq \lambda_1(\Theta(z'; \zeta))\simeq |\zeta^m+\sum_{k=1}^m c_k(z')\zeta^{m-k}|^2. 
\end{equation}
The coefficients $c_k$ are holomorphic functions on $B^{n-1}$ that vanish at $z'=0$, and \begin{equation}\label{eq:2mleqv}2m\leq \mathfrak{h}_p(\varphi).\end{equation} 
	\end{prp}
Notice how we rescaled the small ball $B$ and the disc $D(0,\tau)$ to unit size, as we may (at the cost of modifying the implicit constants). 

\subsection{Scaling}\label{sec:scaling}

We now exploit the fact that the complex Hessian $H(z)$ is homogeneous of degree $2d-2$: \[
(H(Rz)v,v) = R^{2d-2}(H(z)v,v)\qquad \forall R>0,\ z\in \C^n, \ v\in \C^n, 
\]
and that the eigenvalues $\lambda_j(z)$ are $\C$-homogeneous functions of degree $2d-2$. 

By homogeneity, we may assume without loss of generality that the point $p$ in the second item of Lemma \ref{lem:main_estimate} has unit distance from the origin, i.e., $p\in \sfera^{2n-1}$. 

Let $\Theta:B^{n-1}\times \mathbb{D}\rightarrow W$ be the biholomorphism and let \begin{equation}\label{weierstrass}
P_{z'}(\zeta)=\zeta^m+\sum_{k=1}^mc_k(z')\zeta^{m-k}
\end{equation} be the Weierstrass polynomial of Proposition \ref{prp:preparation}. We may assume that every root of $P_{z'}$ is contained in $D\left(0,\frac{1}{2}\right)$ and that $W\subseteq \{1/2< |z|< 3/2\}$. This is easily achieved by a suitable change of variables in $z'$ and $\zeta$. 

We choose $\gamma>1$ and an open spherical cap $K\subset W\cap S^{2n-1}$ centered at $p$ with the property that \begin{equation}\label{cap_inside}
	\{t q\colon \ 1\leq t< \gamma,\ q\in K \}\subseteq W.
\end{equation}
The collection of open sets $\{W_k:=\gamma^kW\}_{k\geq 0}$ has bounded overlapping and covers $V:=\{tq\colon\ t\geq 1, \ q\in K\}$, that is, \begin{equation}\label{bounded_overlapping}
1_V\leq \sum_{k\geq 0} 1_{W_k}\lesssim 1.
\end{equation}

The next thing we are going to do is rescale the biholomorphism $\Theta$ to adapt it to the sets $W_k$. Let \[
\Theta_k(z';\zeta):=\gamma^k\Theta(z';\zeta)\qquad (k\in \N), 
\]
which defines a biholomorphism of $B^{n-1}\times \mathbb{D}$ onto $W_k$. We have the simple change of variables identity \begin{equation}\label{change_of_variables}
\int_{W_k} g\,  d\mathcal{H}_{2n}\simeq \gamma^{2kn}\int_{B^{n-1}}\left(\int_{\mathbb{D}} g\circ\Theta_k(z';\zeta)\,  d\mathcal{H}_2(\zeta)\right)d\mathcal{H}_{2n-2}(z'), 
\end{equation}
for every $g:W_k\rightarrow \C$ (say, nonnegative and measurable). Here $\mathcal{H}_m$ is the $m$-dimensional (Lebesgue, or Hausdorff) measure and the implicit constant is uniform in $k$. 

Now that the stage is set, we come to the proof of \eqref{eq:main_estimate}. The local contribution to our energy form coming from $W_k$ is \[
\E^\varphi_{W_k}(w):=\int_{W_k} |\nabla^{0,1}w|^2e^{-2\varphi}+2\int_{W_k}\lambda_1 |w|^2 e^{-2\varphi}\qquad (w\in C^\infty_c(\C^n)). 
\]
Since $\Theta_k(z';\zeta)=(\Theta_{k,1}(z'; \zeta), \ldots, \Theta_{k,n}(z'; \zeta))$ is holomorphic, we have\[
\frac{\partial}{\partial \overline\zeta}(w\circ \Theta_k)= \sum_{j=1}^n\frac{\partial w}{\partial \overline{z}_j} \circ \Theta_k\, \overline{\frac{\partial \Theta_{k,j}}{\partial\zeta}},\] and Cauchy--Schwarz gives
\[
\left|\frac{\partial}{\partial \overline\zeta}(w\circ \Theta_k)\right|\lesssim \gamma^k |\nabla^{0,1}w|\circ \Theta_k.
\] 
Hence, changing variables with \eqref{change_of_variables}, we get
\begin{eqnarray*}
	\E^\varphi_{W_k}(w)&\simeq & \gamma^{2kn}\int_{B^{n-1}}\int_{\mathbb{D}} |\nabla^{0,1}w|^2\circ \Theta_k\,  e^{-2\varphi\circ\Theta_k}+\gamma^{2kn}\int_{B^{n-1}}\int_{\mathbb{D}}\lambda_1\circ\Theta_k\,  |w\circ \Theta_k|^2 e^{-2\varphi\circ \Theta_k}\\
	&\gtrsim& \gamma^{2k(n-1)}\int_{B^{n-1}}\left\{\int_{\mathbb{D}} \left|\frac{\partial}{\partial \overline\zeta}(w\circ \Theta_k)\right|^2  e^{-2\varphi\circ \Theta_k}+\int_{\mathbb{D}}\gamma^{2k}\lambda_1\circ \Theta_k\,  |w\circ \Theta_k|^2 e^{-2\varphi\circ \Theta_k}\right\}\\
	&=& \gamma^{2k(n-1)}\int_{B^{n-1}}\left\{\int_{\mathbb{D}} \left|\frac{\partial}{\partial \overline\zeta}(w\circ \Theta_k)\right|^2  e^{-2\gamma^{2kd}\varphi\circ \Theta}+\int_{\mathbb{D}}\gamma^{2kd}\lambda_1\circ \Theta\,  |w\circ \Theta_k|^2 e^{-2\gamma^{2kd}\varphi\circ \Theta}\right\}\\
	&\simeq& \gamma^{2k(n-1)}\int_{B^{n-1}}\left\{\int_{\mathbb{D}} \left|\frac{\partial}{\partial \overline\zeta}(w\circ \Theta_k)\right|^2  e^{-2\gamma^{2kd}\varphi\circ \Theta}+\int_{\mathbb{D}}\Delta_\zeta(\gamma^{2kd}\varphi\circ \Theta)\,  |w\circ \Theta_k|^2 e^{-2\gamma^{2kd}\varphi\circ \Theta}\right\}, 
\end{eqnarray*}
where in the third line we used the homogeneity of $\varphi$ and $\lambda_1$, and in the last step we used \eqref{eq:preparation}. 

We are now in a position to apply the Main one-dimensional estimate (Proposition \ref{lem:main_1d_estimate}). Since the leading coefficient of the rescaled Weierstrass polynomial is $A=\gamma^{kd}$, we get \[
\int_{\mathbb{D}} \left|\frac{\partial}{\partial \overline\zeta}(w\circ \Theta_k)\right|^2  e^{-2\gamma^{2kd}\varphi\circ \Theta}+\int_{\mathbb{D}}\Delta_\zeta(\gamma^{2kd}\varphi\circ \Theta)\,  |w\circ \Theta_k|^2 e^{-2\gamma^{2kd}\varphi\circ \Theta}\gtrsim \gamma^{\frac{2kd}{m+1}}\int_{\mathbb{D}}  |w\circ \Theta_k|^2 e^{-2\gamma^{2kd}\varphi\circ \Theta}, 
\]
if $k\geq k_0$ is large enough. Plugging this into the lower bound on $\E^\varphi_{W_k}(w)$, we get 
\begin{eqnarray*}
\E^\varphi_{W_k}(w)\gtrsim \gamma^{2k\left(\frac{d}{m+1}-1\right)}\gamma^{2kn}\int_{B^{n-1}}\left\{\int_{\mathbb{D}}  |w\circ \Theta_k|^2 e^{-2\varphi\circ \Theta_k}\right\}\simeq \gamma^{2k\left(\frac{d}{m+1}-1\right)}\int_{W_k}|w|^2e^{-2\varphi}, 
\end{eqnarray*}
again for $k\geq k_0$. Since $|z|\simeq \gamma^k$ on $W_k$, summing over $k\geq k_0$, and recalling \eqref{eq:2mleqv} and \eqref{bounded_overlapping}, we finally obtain \eqref{eq:main_estimate}. This completes the proof of Theorem \ref{thm:main}. 

\section{Applications of Theorem \ref{thm:main}}\label{sec:examples} 

\subsection{Case $n=1$} The application of Theorem \ref{thm:main} to this case is particularly straightforward. Indeed, there is (up to scalar multiples) only one  $\varphi\in \mathcal{H}_{1,d}$ for every $d$, and $\widetilde{F}$ is the identity map of $\mathbb{P}^0$. We obtain the following corollary (already covered by the much more general result of Kohn and Greiner). 

\begin{cor}
The domain \[
	\Omega=\{(z_1,z_2)\in \C^2\colon\, \im(z_2)>|z_1|^{2d}\}
	\] 	satisfies a subelliptic estimate of order $\frac{1}{2d}$ at the origin, and no better one. 
\end{cor}

Notice that in this case the only nontrivial part of the crucial Lemma \ref{lem:main_estimate} is the local estimate, and all we need of the machinery developed in the paper is the basic $\dbar$-UP.

\subsection{Case $n=2$} 

In this case $\widetilde{F}$ is a holomorphic self-map of $\mathbb{P}^1$. It is obvious that $\dim\ker d\widetilde{F}(q)\leq 1$ for every $q\in \mathbb{P}^1$, and it is easy to see that $\mathfrak{t}(\widetilde{F})\leq d$. Indeed, this amounts to the fact that $\widetilde{F}$ is a branched cover of degree $d$, as may be seen representing the map in appropriate local parameters. Thus, Theorem \ref{thm:main} gives the following corollary. 

\begin{cor}
Let $F_1(z_1,z_2),F_2(z_1,z_2)$ be homogeneous polynomials of degree $d\geq 1$ without common zeros outside the origin. Then the domain \[
\Omega=\{(z_1,z_2,z_3)\in \C^3\colon\, \im(z_3)>|F_1(z_1,z_2)|^2+|F_2(z_1,z_2)|^2\}
\] 	satisfies a subelliptic estimate of order $\frac{1}{2d}$ at the origin, and no better one. 
	\end{cor}

\subsection{Cases $n\in \{3,4\}$} Things get more subtle and interesting when $n=3$ or $n=4$. Let us first show that examples exist in abundance. 

\begin{prp}\label{prp:generic_new}
	Let $d\geq 1$. Then the generic holomorphic self-map of  $\mathbb{P}^2$ of degree $d$ satisfies \eqref{eq:one_dim_ker_bis}. The same is true of the generic holomorphic self-map of $\mathbb{P}^3$ of degree $d$. 
\end{prp}

\begin{proof}
Fix $n,d\geq 1$. Let $N:={n-1+d \choose n-1}$ be the number of monomials of degree $d$ in $n$ variables and let $v:\mathbb{P}^{n-1}\rightarrow \mathbb{P}^{N-1}$ be the Veronese embedding \[
[(z_1,\cdots, z_n)]\mapsto [(z_1^{\alpha_1}\cdots z_n^{\alpha_n})_{\alpha_1+\cdots+\alpha_n=d}].\] Denote by $V$ the preimage of $v(\mathbb{P}^{n-1})$ under the canonical projection $\C^n\setminus \{0\}\rightarrow \mathbb{P}^{n-1}$. 

Every $n\times N$ matrix $A$ induces a projection $\widetilde{A}:\mathbb{P}^{N-1}\rightarrow \mathbb{P}^{n-1}$, which is regular on the complement of a projective subspace, which is of codimension $n$ if $A$ has full-rank. Since $v(\mathbb{P}^{n-1})$ has dimension $n-1$, the generic $A$ of full-rank is such that $\widetilde{A}$ is regular on $v(\mathbb{P}^{n-1})$, and hence induces a self-map $\widetilde{F}=\widetilde{A}\circ v$ of $\mathbb{P}^{n-1}$. Of course, every self-map of $\mathbb{P}^{n-1}$ of degree $d$ arises in this way. Let $\mathcal{A}\subseteq \C^{n\times N}$ be the Zariski open set of matrices $A$ of full-rank such that $\widetilde{A}$ is regular on $V$, i.e., \[
\mathcal{A}=\{A\in \C^{n\times N}\colon\, A \textrm{ has full rank and }\ker A\cap V=\emptyset\}.
\] Let \[
\ker :\mathcal{A}\rightarrow \mathrm{Gr}_{N-n}(\C^N)
\] be the morphism $A\mapsto \ker A$, where $\mathrm{Gr}_k(U)$ is the Grassmannian of $k$-planes in the complex vector space $U$. 

It is easily seen that the self-map of $\mathcal{P}^{n-1}$ induced by $A\in \mathcal{A}$ admits a point where the null-space of the Jacobian has dimension $\geq 2$ if and only if there exists $p\in V\subseteq \C^n$ and $L\in \mathrm{Gr}_2(T_p V)$ such that $\ker A\supseteq L$. Let then \[
\mathcal{B}=\{M\in \mathrm{Gr}_{N-n}(\C^N)\colon\, \exists p\in V, \, L\in \mathrm{Gr}_2(T_pV)\, \textrm{s.t.}\, M\supseteq L\}.
\]
The thesis will be established if we show that for $n\in \{3,4\}$ the complement of $\mathcal{B}$ has nontrivial Zariski interior, because the same will then be true of the complement of $\ker^{-1}(\mathcal{B})$. 

Consider the variety \[
Z=\{(p,L,M)\colon\, p\in V,\, L\in \mathrm{Gr}_2(T_pV), \, M\in \mathrm{Gr}_{N-n}(\C^N)\, \textrm{s.t.}\, M\supseteq L\}
\] and denote by $Z_{12}$ its projection onto the first two factors and by $Z_{23}$ its projection onto the second and third factors. Notice that since $V$ is invariant under scalings $v\mapsto \lambda v$ for every $\lambda\in \C^\times$, the projection onto the second factor $p_2: Z_{12}\rightarrow \mathrm{Gr}_2(\C^N)$ has the property that \[
\dim p_2(Z_{12})\leq \dim Z_{12}-1 = n+2(n-2)-1=3n-5, 
\]
where we used the fact that the $V$ has dimension $n$ and the formula $\dim \mathrm{Gr}_k(U)=k(\dim U-k)$. It is now easy to see that \begin{eqnarray*}
\dim Z_{23}&=&\dim p_2(Z_{12})+\dim \mathrm{Gr}_{N-n-2}(\C^{N-2})\\
&\leq  &3n-5+(N-n-2)n=(N-n)n+n-5. 
\end{eqnarray*}
Here we used the fact that $M\in \mathrm{Gr}_{N-n}(\C^N)$ is such that $M\supseteq L$ if and only if $M/L\in \mathrm{Gr}_{N-n-2}(\C^N/L)$ when $L$ is a $2$-plane. 

Since the projection onto the third factor of $Z$ is exactly $\mathcal{B}$, it follows that $\dim \mathcal{B}\leq \dim Z_{23}\leq (N-n)n+n-5$. If $n<5$, then this number is smaller than the dimension of $\mathrm{Gr}_{N-n}(\C^N)$, and the thesis follows. 

	\end{proof}

\begin{rmk}
As was pointed out to us by J.~Silverman on MathOverFlow, in the two (projective) dimensional case the critical locus of a self-map of projective space is generically smooth (see \cite[Theorem 15 (a)]{silverman}). This is a stronger statement than our Proposition \ref{prp:generic_new}. In fact, 
any point $q\in \mathbb{P}^2$ such that $\dim \ker d\widetilde{F}(q)\geq 2$ is a singular point of $Z=\{\dim \ker d\widetilde{F}\geq 0\}$. This follows from Jacobi's formula for the differential of the determinant \[
	d \det (A)  =  \mathrm{tr}(\mathrm{adj}(A)dA)\qquad (A\in \C^{n\times n}), 
	\]
	where $\mathrm{adj}(A)$ is the transpose of the cofactor matrix, which is the zero matrix if $\dim \ker A\geq 2$, and $dA$ is the matrix of differentials of the components of $A$. 

\end{rmk}

Next, we notice that the invariant $\mathfrak{t}(\widetilde{F})$ of a holomorphic self-map of $\mathbb{P}^n$ need not be $\leq d$. The simplest examples are in dimension $n=2$ and degree $d=2$. E.g., consider the map 
\[
F(z)=(z_1^2+z_2z_3, z_2^2+z_1z_3, z_3^2), 
\]
whose Jacobian matrix is 
\[
\det J(z) = \det \begin{pmatrix}
	2z_1& z_3& z_2\\
	z_3& 2z_2& z_1\\
	0& 0& 2z_3
\end{pmatrix} = 2z_3(4z_1z_2-z_3^2)
\]
So the critical locus $Z$ of $\widetilde{F}$ is the union of a line and a conic. It is an easy computation to check that the rank of $J(z)$ is $2$ at every point of $Z$.
Hence, our mapping satisfies the hypothesis \eqref{eq:one_dim_kernel} of Theorem \ref{thm:main}. 

By elementary jet computations, one may show the following facts: \begin{enumerate}
\item If $p=[(z_1,z_2,z_3)]$ belongs to the ramification locus of $\widetilde{F}$ and it is different from
\[[(1,0,0)],\ [(0,1,0)],\ [(1,1,2)],\ [(1,\xi,2\xi^2)],\ [(1,\xi^2,2\xi)]\]
with $\xi$ a primitive $3$rd rooth of unit, then for every nonsingular analytic discs $\widetilde{\psi}$ with $\widetilde{\psi}(0)=p$ we have $\nu_0(\tilde{F}\circ \widetilde{\psi})\leq2$ and equality is attained.

\item If $p$ is $[(1,1,2)]$, $[(1,\xi,2\xi^2)]$, or $[(1,\xi^2,2\xi)]$ (which lie on the conic, but not on the line), then for every nonsingular analytic disc $\widetilde{\psi}$ with $\widetilde{\psi}(0)=p$ we have $\nu_0(\widetilde{F}\circ\widetilde{\psi})\leq3$ and equality is attained. For example, consider $\psi(\zeta)=(1+\xi\zeta,\xi^2-\zeta-\xi\zeta^2/2,2\xi)$ and set $\widetilde{\psi}=\pi\circ\psi$.

\item If $p$ is $[(1,0,0)]$ or $[(0,1,0)]$ (which are the intersection points of the line and the conic), then for every nonsingular analytic disc $\widetilde{\psi}$ for which $\widetilde{\psi}(0)=p$, we have $\nu_0(\widetilde{F}\circ\widetilde{\psi})\leq4$ and equality is attained. For example, consider $\psi(\zeta)=(1+\zeta^3/2,\zeta,-\zeta^2)$, then
\[
F\circ\psi(\zeta)=(1+\zeta^4/6,-\zeta^5/2,\zeta^4)
\]
that is $F$ has order $4$ along a nonsingular curve through the point $(1,0,0)$, hence setting $\widetilde{\psi}=\pi\circ\psi$, we obtain the claimed equality.
\end{enumerate}

Therefore, we have $\mathfrak{t}(\widetilde{F})=4$ and Theorem \ref{thm:main} allows us to conclude that
$$s(\Omega;0)=\frac{1}{8}$$
where
$$\Omega=\{(z,z_{n+1})\in\C^{n+1}\ \colon\ \im z_{n+1}>|z_1^2+z_2z_3|^2+|z_2^2+z_1z_3|^2+|z_3^2|^2\}\;.$$

The discussion of this example shows how the sharp order of subellipticity may be actually computed for the (germs of) domains covered by Theorem \ref{thm:main}. 

\subsection{Case $n\geq 5$} Let's prove that Theorem \ref{thm:main} is vacuous in dimension $n\geq 5$. 

\begin{prp}
	Let $m\geq 4$. Every regular morphism $\widetilde{F}:\mathbb{P}^m\rightarrow \mathbb{P}^m$ of degree $d\geq 2$ admits a point $q$ where $\dim \ker d\widetilde{F}(q)\geq 2$. 
\end{prp}

\begin{proof} 
	Let $Y_2$ be the subvariety of the space of square matrices $\C^{n \times n}$ given by the condition $\mathrm{rk} M\leq n-2$. We claim that its codimension is $4$. 
	
	To see this, consider the complex manifold
	\[Z_2=\{(M,V)\in \C^{n \times n}\times\mathrm{Gr}_{n-2}(\C^n)\ :\ Mv\in V\quad \forall v\in \C^n \}.\] Notice that $Z_2$ is a vector bundle on $\mathrm{Gr}_{n-2}(\C^n)$ with fibers of complex dimension $n(n-2)$. Since the dimension of $\mathrm{Gr}_{n-2}(\C^n)$ is $2(n-2)$, we have $\dim_\C Z_2=n^2-4$. Since $Z_2$ is a desingularization of $Y_2$, it must have the same dimension $n^2-4$. This proves the claim. Clearly, the projectivization $\mathbb{P}(Y_2)$ in $\mathbb{P}^{n^2-1}$ has also codimension $4$.  
	
	Given $F:\C^n\rightarrow \C^n$ homogeneous of degree $d\geq 2$, its Jacobian $J_F = \left(\frac{\partial F_j}{\partial z_k}\right)_{j,k}:\C^n\rightarrow \C^{n\times n}$ induces a map $\widetilde{J}_F:\mathbb{P}^{n-1}\rightarrow \mathbb{P}^{n^2-1}$. For a generic choice of $F$, the image of $\widetilde{J}_F$ is a subvariety of dimension $n-1$. To see this, it is enough to exhibit an example in every degree $d\geq 2$, e.g., $F=(z_1^d,\ldots, z_n^d)$. Therefore, if $n-1\geq 4$, $\widetilde{J}_F(\mathbb{P}^{n-1})$ must intersect $\mathbb{P}(Y_2)$. This proves that the statement is true for a generic choice of $F$. 
	
	To conclude, let $\mathbb{P}^{N(d,n)}$ the projectivized of the space of $n$-tuples of homogeneous polynomials of degree $d$ in $n$ variables. The set \[
	W = \{([p], [F])\in \mathbb{P}^{n-1}\times \mathbb{P}^{N(d,n)}\colon\, \mathrm{rk}(J_F(p))\leq n-2\}
	\]
	is Zariski-closed. Thus, the same is true of its projection onto the second factor. Since we proved above that this projection contains a Zariski-open subset, the thesis follows. \end{proof}

\section{Closing remarks}\label{sec:closing}

A few additional comments may help to clarify the scope of our method. First of all, we only concern ourselves with rigid domains, since translation invariance in $\re(z_{n+1})$ allows one to exploit Fourier analysis, which shows that \eqref{eq:subelliptic} is equivalent to a certain one-parameter family of estimates. Here the parameter is of course the frequency variable "dual" to $\re(z_{n+1})$. After some rather standard manipulations, this family of estimates produces a sufficient condition for subellipticity involving the quadratic forms $\mathbf{E}^{\xi\varphi}$, that is, Proposition \ref{prp:spectral_gap}. 

Secondly, while a $\bar\partial$-uncertainty principle might eventually turn out to be useful for more general rigid domains, it seems to be already highly non-trivial to put together the local information it provides and to obtain the desired spectral gap estimates in the narrower class of special domains \eqref{eq:special_domain}, whose pseudoconvex geometry reduces to the algebraic geometry of the polynomial mapping $F$. Proving good lower bounds for the sharp order of subellipticity for special domains (e.g., proving D'Angelo's conjecture or something of comparable strength) is already a very desirable (and apparently difficult) goal, and so it seems reasonable to think of the category of special domains as a wide enough arena for our investigations. 

The natural question is then: what is still to be done in order to prove "good" lower bounds for $s(\Omega, 0)$ when $\Omega$ is a general special domain (of D'Angelo finite type), thus completing our program of proving subellipticity via appropriate $\bar\partial$-uncertainty principles? We provide below a couple of remarks in this regard. 

The  "local hard analysis" in this paper happens in dimension one (and in Section \ref{sec:dbar_up}). The passage from the local one-dimensional analysis to the global $n$-dimensional one is accomplished in two steps: an argument with the approximate minimal eigenvector fields of Section \ref{sec:AME} yields a "single-scale" $n$-dimensional estimate, and a scaling argument then allows to deduce the final global result from the single-scale estimate. We are able to make the first step work under the assumption that \emph{the null space of the Jacobian of $F$ is one-dimensional at critical points outside the origin}; the most expedient way to carry out the second step is instead to assume that $F$ is \emph{homogeneous} (that is, its components are homogeneous polynomials of the same degree). We can check that the two assumptions co-exist (generically) only in dimension $n\leq 4$ (hence, for domains in dimension $n+1\leq 5$). The resulting class of low-dimensional rigid domains is the one that we use as test-bed for the whole approach in Section \ref{sec:ker_one_dim}. Of the two assumptions that we make in Section \ref{sec:ker_one_dim} (one-dimensionality of the null-space of the Jacobian and homogeneity), the first one is the more serious one. We believe in fact that the second may be relaxed (if not dispensed with), while dealing with higher dimensional null-spaces should require new ideas.

\appendix

\section{}\label{sec:appendix_A}

A harmonic function on a disc can be represented as the real part of a holomorphic function; Proposition \ref{lem:well_known} below states that a function on a disc whose Laplacian is bounded can be represented as the real part of a holomorphic function plus a controlled error. The proof of Proposition \ref{lem:well_known} uses another quite elementary fact. 

\begin{prp}[Gradient estimate for holomorphic functions]\label{prop:well_known}\[
	||F'||_{L^\infty(D(z,0.r))}\lesssim r^{-2}||\re(F)||_{L^2(D(z,r))}\qquad \forall  F\in \mathcal{O}(D(z,r)).
	\]
	
\end{prp}

\begin{proof}
	Let $B$ be the Bergman projector on the unit disc $\mathbb{D}$. Let $F(\zeta)$ be a square integrable holomorphic function on $\mathbb{D}$. By the mean value property, $\overline{F(\zeta)}-\overline{F(0)}$ is orthogonal to the Bergman space and hence $B(\overline{F(\zeta)}-\overline{F(0)})=0$. This identity may be rewritten as \[
	F(\zeta)+\overline{F(0)}=2B\left(\re F(\zeta)\right). 
	\]
	Differentiating, we get \[
	F'(\zeta)=2\partial_\zeta B\left(\re F(\zeta)\right). 
	\]
	Ellipticity (or the Cauchy integral) yields \[
	||F'||_{L^\infty(0.5\mathbb{D})}\lesssim ||B\left(\re F\right)||_{L^2(\mathbb{D})}\leq ||\re F||_{L^2(\mathbb{D})}. 
	\]
	If $F$ is an arbitrary holomorphic function on $\mathbb{D}$, applying this identity to $F(t\zeta)$ ($t<1$), we get $t||F'||_{L^\infty(0.5t\mathbb{D})}\lesssim t^{-1}||\re F||_{L^2(t\mathbb{D})}$. Letting $t$ tend to $1$, the monotone convergence theorem allows to deduce the inequality of the statement for every holomorphic $F$, when $D(z,r)=\mathbb{D}$. Traslation invariance and scaling gives the full thesis. 
\end{proof}

\begin{prp}[Bounded curvature implies almost harmonicity]\label{lem:well_known}
	Let $\varphi\in C^2(D(z,r); \R)$ be such that $|\Delta \varphi(\zeta)|\leq Cr^{-2}$ for all $\zeta\in D(z,r)$. Then there exists a holomorphic function $G:D(z,r)\rightarrow \C$ such that \[
	\varphi = \re(G)+O(C), \quad r||G'||_{L^\infty(D(z,0.5r))}\lesssim ||\varphi||_{L^\infty(D(z,r))}+C
	\]
\end{prp}

\begin{proof}
	It is enough to prove the case $z=0$ and $r=1$, the general case follows by translation invariance and scaling. Define \[
	b(\zeta):=\begin{cases} \Delta\varphi(\zeta)\quad &\zeta\in \mathbb{D}\\
		0& \zeta\notin \mathbb{D}
	\end{cases}
	\] Let $b*\Gamma$ be its Newtonian potential, where $\Gamma(\zeta)=\frac{1}{2\pi}\log|\zeta|$. Then $\varphi-b*\Gamma$ is harmonic on $\mathbb{D}$, and hence there exists $G\in \mathcal{O}(\mathbb{D})$ such that $\varphi=\re(G)+b*\Gamma$. Since $\Gamma$ is uniformly integrable on discs of radius one, we have $|b*\Gamma|\lesssim C$. Hence, Proposition \ref{prop:well_known} gives \[
	||G'||_{L^\infty(0.5\mathbb{D})}\lesssim ||\re(G)||_{L^\infty(\mathbb{D})}\lesssim ||\varphi||_{L^\infty(\mathbb{D})}+C. 
	\]\end{proof}

\bibliography{subelliptic}
\bibliographystyle{alpha}

\end{document}